\documentclass[english]{article}

\usepackage{amsxtra}
\usepackage{mathtools}
\usepackage[utf8]{inputenc}
\usepackage{amsfonts}
\usepackage[toc,page]{appendix}
\usepackage{amsthm}
\usepackage{graphicx}
\usepackage{xcolor}
\usepackage{amssymb}
\usepackage{mathrsfs}
\usepackage{bm}
\usepackage{comment}
\usepackage{amsmath}
\usepackage[hyphens]{url} 
\usepackage[colorlinks=true]{hyperref}
\usepackage{titlesec}
\usepackage{tikz-cd}
\usepackage{textcomp}
\usepackage{fontawesome}
\usepackage{mathabx}
\usepackage[nottoc,numbib]{tocbibind}
\usepackage[bbgreekl]{mathbbol}
\usepackage{relsize}
\usepackage[T1]{fontenc}
\usepackage{charter}
\usepackage{seqsplit}
\usepackage{breqn}
\usepackage{upgreek}
\usepackage{mathrsfs}
\usepackage{authblk}

\theoremstyle{plain}
\newtheorem{thm}{Theorem}[section]
\newtheorem{cor}[thm]{Corollary}
\newtheorem{pro}[thm]{Proposition}
\newtheorem{lem}[thm]{Lemma}
\theoremstyle{definition}
\newtheorem{dfn}[thm]{Definition}
\newtheorem{example}[thm]{Example}
\newcommand{\Ghat}{\widehat{G}}
\theoremstyle{remark}
\newtheorem{rmk}[thm]{Remark}

\newtheorem{conjecture}[thm]{Conjecture}
\newtheorem*{conjecture*}{Conjecture}
\newtheorem*{assumption*}{Assumption}

\newtheorem{construction}[thm]{Construction}
\newtheorem{setting}[thm]{Setting}

\DeclareSymbolFontAlphabet{\mathbb}{AMSb}
\DeclareSymbolFontAlphabet{\mathbbl}{bbold}

\def\Hom{\mathop{\textup{Hom}}\nolimits}

\def\Nt{\mathop{\textup{Nt}}\nolimits}
\def\BBW{\mathop{\textup{BW}}\nolimits}

\def\Gal{\mathop{\textup{Gal}}\nolimits}
\def\End{\mathop{\textup{End}}\nolimits}

\def\Perf{\mathop{\bf Perf}\nolimits}

\def\perf{\mathop{\textup{perf}}\nolimits}

\def\Spec{\mathop{\bf Spec}\nolimits}
\def\Spa{\mathop{\bf Spa}\nolimits}
\def\Spf{\mathop{\bf Spf}\nolimits}
\def\Spd{\mathop{\bf Spd}\nolimits}
\def\Rep{\mathop{\bf Rep}\nolimits}
\def\Sht{\mathop{\bf Sht}\nolimits}
\def\Coh{\mathop{\bf Coh}\nolimits}
\def\Isoc{\mathop{\bf Isoc}\nolimits}
\def\Shv{\mathop{\bf Shv}\nolimits}
\def\IndShv{\mathop{\bf IndShv}\nolimits}
\def\Loc{\mathop{\bf Loc}\nolimits}
\def\CohSpr{\mathop{\bf CohSpr}\nolimits}
\def\Flex{\mathop{\bf Flex}\nolimits}
\def\Igs{\mathop{\bf Igs}\nolimits}
\def\Ig{\mathop{\bf Ig}\nolimits}

\def\Div{\mathop{\bf Div}\nolimits}
\def\Bun{\mathop{\bf Bun}\nolimits}
\def\ULA{\mathop{\bf ULA}\nolimits}

\def\can{\mathop{\textup{can}}\nolimits}
\def\unip{\mathop{\textup{unip}}\nolimits}
\def\spec{\mathop{\textup{spec}}\nolimits}

\def\Adm{\mathop{\rm Adm}\nolimits}

\def\IndCoh{\mathop{\bf IndCoh}\nolimits}

\def\GL{\mathop{\textup{GL}}\nolimits}
\def\Ad{\mathop{\textup{Ad}}\nolimits}
\def\GSP{\mathop{\textup{GSP}}\nolimits}
\def\Fib{\mathop{\textup{Fib}}\nolimits}
\def\Gr{\mathop{\textup{Gr}}\nolimits}
\newcommand{\Hck}{\bm{\mathscr{H}}\!\bm{ck}}

\def\Res{\mathop{\rm Res}\nolimits}

\def\GU{\mathop{\textup{GU}}\nolimits}

\def\et{{\textup{\acute et}}}
\def\red{{\textup{red}}}
\def\ad{{\textup{ad}}}

\def\Sh{\mathop{\bf Sh}\nolimits}

\newcommand{\BA}{{\mathbb{A}}}

\newcommand{\BF}{{\mathbb{F}}}

\newcommand{\BQ}{{\mathbb{Q}}}

\newcommand{\BZ}{{\mathbb{Z}}}

\newcommand{\CA}{{\cal A}}
\newcommand{\CB}{{\cal B}}

\newcommand{\CD}{{\cal D}}
\newcommand{\CE}{{\cal E}}

\newcommand{\CG}{{\cal G}}
\newcommand{\CH}{{\cal H}}
\newcommand{\CI}{{\cal I}}

\newcommand{\CM}{{\cal M}}

\newcommand{\CO}{{\cal O}}
\newcommand{\CP}{{\cal P}}

\newcommand{\CS}{{\cal S}}

\newcommand{\CY}{{\cal Y}}

\title{Igusa stack for some exceptional Shimura varieties}

\date{\today}

\author[1]{Ali Partofard\thanks{\scriptsize \href{mailto:hadi@ipm.ir}{alipartofard@ipm.ir}}}

\begin{document}

\maketitle
\begin{abstract}
We study the integral models of meta-unitary Shimura varieties through the lens of Scholze’s fiber product conjecture. Reformulating Bültel’s original construction in terms of moduli stacks of Shtukas and Igusa stacks, we prove the validity of the fiber product formula for this class of non-abelian type Shimura varieties, thereby generalizing the works of Zhang and Daniels, Van Hoften, Kim, and Zhang. We utilize this geometric description to derive local-global compatibility results and, adapting the strategy of Zhu and Yang, apply the unipotent categorical local Langlands correspondence to prove a general vanishing theorem for the generic part of the cohomology of meta-unitary Shimura varieties.
\end{abstract}

\tableofcontents
\section{Introduction}
Let $(G,\mu)$ be a Shimura datum and let $p$ be a fixed prime number. A central problem in the theory of Shimura varieties, as articulated in part of the Langlands--Rapoport conjecture, is the construction of "nice" integral models at hyperspecial level. By a \emph{nice} integral model, we mean one that is smooth over the ring of integers, satisfies a suitable extension property, and whose special fiber admits a group-theoretic description reflecting the structure of the Shimura datum. Understanding these integral models is crucial both for arithmetic applications and for the study of the reduction of Shimura varieties modulo $p$.

The traditional approach to constructing integral models is to embed a Shimura variety into a Siegel modular variety and define its integral model as the closure inside the integral model of the Siegel variety. Using this method, one can construct integral models for all Hodge-type Shimura varieties, and, via Deligne's formalism, even for all abelian-type Shimura varieties\cite{kisin2010integral}. However, this approach has limitations: it does not extend to all Shimura varieties, and it provides little direct information about the $p$-adic geometry of the resulting integral models.

Motivated by Mantovan's formula, Scholze's fiber product conjecture proposes an alternative way to describe—or even construct—Shimura varieties using Shtukas and Igusa stacks. Roughly speaking, the Igusa stack describes the special fiber of the integral model, while the stack of Shtukas controls the deformation to mixed characteristic. This perspective has the advantage of directly connecting the integral model to key objects in $p$-adic geometry, such as the stack $\Bun_G$, the affine Grassmannian, and the integral local Shimura variety.

Mingjia Zhang proved the Scholze fiber product conjecture for certain PEL-type Shimura varieties\cite{zhang2023pel}, and later Daniels, Van Hoften, Kim, and Zhang generalized it to all Hodge-type Shimura varieties\cite{daniels2024Igusa}. This construction has several immediate consequences, including vanishing results for the cohomology of Hodge-type Shimura varieties and local-global compatibility between the cohomology of Hodge-type Shimura varieties and the Scholze–Fargues local Langlands correspondence.

Building on this geometric framework, Zhu and Yang\cite{yang2025generic} employed the fiber product description in conjunction with the unipotent categorical local Langlands correspondence\cite{hemo2021unipotent} to establish a general vanishing theorem for the generic part of the cohomology of Shimura varieties of Abelian type.

In 2008, B\"ultel constructed an integral model for a class of Shimura varieties called \emph{meta-unitary Shimura varieties}\cite{bultel2008pel}, which are not necessarily of abelian type. Remarkably, this construction anticipates the philosophy of Scholze's fiber product conjecture: he first constructs the special fiber of the Shimura variety and then defines the integral model as a fiber product between this special fiber and the stack of $(G,\mu)$-displays. In this paper, among other things, we reformulate B\"ultel's construction in the modern language of Shtukas and integral local Shimura varieties.

The aim of this paper is to prove the Scholze fiber product conjecture and establish other results of Daniels, Van Hoften, Kim, and Zhang in the setting of meta-unitary Shimura varieties; furthermore, we employ the method of Zhu and Yang to establish the vanishing of the generic part of the cohomology.

\subsection{Scholze Fiber Product Conjecture}

Let us now recall the Scholze fiber product conjecture in more detail. We present the conjecture in the formulation given by Zhang \cite{zhang2023pel}, Let $(G,\mu)$ be a Shimura datum, and denote by $[\mu^{-1}]$ the conjugacy class of minuscule cocharacters associated with $(G,\mu)$, with reflex field $E_0$. Fix a rational prime $p$, and let $E$ be the completion of $E_0$ at a place above $p$. Choose a compact open subgroup $K = K_p K^p \subset G(\mathbb{A}_f)$, where $K_p \subset G(\mathbb{Q}_p)$ and $K^p$ is the level structure away from $p$ and let $K=K_pK^p$. Denote by
$
\Sh_{K}(G,\mu)^\diamondsuit
$
the Shimura variety at level $K$, viewed as a diamond over $\Spd(E)$, and define
$
\Sh_{K^p}(G,\mu)^\diamondsuit := \varprojlim_{K_p} \Sh_{K_pK^p}(G,\mu)^\diamondsuit
$
as the inverse limit over open compact subgroups $K_p$.

On the local side, let $\Gr_G$ denote the $B^+_{\mathrm{dR}}$-affine Grassmannian associated with $G_{\mathbb{Q}_p}$, also viewed as a diamond over $\Spd(E)$. Fix a representative cocharacter $\mu \in [\mu]$, dominant with respect to a chosen Borel and maximal torus over $\overline{\mathbb{Q}}_p$, and denote by $\mathrm{Gr}_{G,\mu^{-1}}$ the corresponding Schubert cell in $\mathrm{Gr}_G$.

Let $\Bun_G$ be the (small) $v$-stack of $G_{\mathbb{Q}_p}$-bundles on the Fargues--Fontaine curve. There are two natural maps:
\[
 \pi_{\mathrm{HT}}\colon \Sh_{K}(G,\mu)^\diamondsuit\longrightarrow \Gr_G,
\qquad
\mathrm{BL}\colon \Gr_G \longrightarrow \Bun_G,
\]
where $\pi_{\mathrm{HT}}$ is the Hodge--Tate period map, whose image lies in $\Gr_{G,\mu^{-1}}$, and $\mathrm{BL}$ is the Beauville--Laszlo morphism sending a modification of the trivial $G$-bundle on the curve $\CY$ to the corresponding $G$-bundle on the Fargues--Fontaine curve.

The Scholze fiber product conjecture asserts that the Shimura variety $\Sh(G,\mu)^\diamondsuit$ can be recovered as a fiber product over $\Bun_G$, in a way that cleanly separates the local $p$-adic geometry (governed by $\Gr_{G,\mu^{-1}}$) from the global data and prime-to-$p$ level structure.

\begin{conjecture}[Scholze, cf.~\cite{zhang2023pel}]\label{mainconjecture}
There exists a small $v$-stacks $\{ \Igs_{K^p,G,\mu} \}$, called \emph{Igusa stack}, equipped with morphisms
\[
\Sh_{K}(G,\mu)^\diamondsuit \xrightarrow{\; \bar\pi_{\mathrm{red}} \;} \Igs_{K^p,G,\mu},
\qquad
\Igs_{K^p,G,\mu} \xrightarrow{\; \bar\pi_{\mathrm{Crys}} \;} \Bun_G,
\]
satisfying the following properties:
\begin{enumerate}
    \item \textbf{Cartesian diagram:} there is a Cartesian (2-)square
    \[
    \begin{tikzcd}
    \Sh_{K^p}(G,\mu)^\diamondsuit \arrow[r, "\pi_{\mathrm{HT}}"] \arrow[d] & \Gr_{G,\mu^{-1}} \arrow[d, "\mathrm{BL}"] \\
    \Igs_{K^p,G,\mu} \arrow[r, "\bar\pi_{\mathrm{HT}}"] & \Bun_{G,\mu^{-1}}
    \end{tikzcd}
    \]
    furthermore for each finite level $K_p\subset G(\mathbb{Q}_p)$ we have the cartesian square:
    \[
    \begin{tikzcd}
    \Sh_{K}(G,\mu)^\diamondsuit \arrow[r, "\pi_{\mathrm{HT}}"] \arrow[d] & {[\Gr_{G,\mu^{-1}}/K_p]} \arrow[d, "\mathrm{BL}"] \\
    \Igs_{K^p,G,\mu} \arrow[r, "\bar\pi_{\mathrm{HT}}"] & \Bun_{G,\mu^{-1}}
    \end{tikzcd}
    \]
    \item \textbf{Hecke action:} The system $\{ \mathrm{Igs}_{K^p,G,\mu} \}$ carries an action of $G(\mathbb{A}_f)$, with $G(\mathbb{Q}_p)$ acting trivially, compatibly with the Hecke actions on the Shimura varieties and period maps. One can recover the Hecke action on the tower of Shimura varieties by the Hecke action on the Igusa stacks and the action of $G(\mathbb{Q}_p)$ on $Gr_{G,\mu^{-1}}$.
    \item \textbf{Minimal compactification:} There exist compactifications $\mathrm{Igs}^*_{K^p}$ extending $\mathrm{Igs}_{K^p}$ such that the Cartesian diagram in (1) extends to include the minimal compactification $S^*_{K^p}$ of the Shimura variety.
    \item \textbf{Integral models:} If $G$ admits a smooth parahoric integral model $\mathcal{G}$ over $\mathbb{Z}_p$, then at level $K_p = \mathcal{G}(\mathbb{Z}_p)$, the Cartesian square in (1) should have an analogue at the level of integral models:
    \[
    \begin{tikzcd}
    \CS_{K,G,\mu}^{\diamondsuit} \arrow[r, "\pi_{\mathrm{Crys}}"] \arrow[d] & \Sht_{G,\mu,K_p} \arrow[d, "\mathrm{BL}"] \\
    \Igs_{K^p,G,\mu} \arrow[r, "\bar\pi_{\mathrm{HT}}"] & \Bun_{G,\mu^{-1}}
    \end{tikzcd}
    \]
    \item \textbf{Functoriality:} The construction of $\Igs_{K^p,G,\mu}$ is natural in the Shimura datum $(G,\mu)$.
\end{enumerate}
\end{conjecture}

Let us elaborate on how the existence of the integral model constructed by B\"{u}ltel is consistent with the general philosophy of Scholze's fiber product conjecture. Assume that the hypothetical Igusa stack $\Igs_G$ exists over the stack $\Bun_G$, and furthermore, assume that it is functorial in $G$. The Igusa stack should admit a stratification by substacks $\Igs_{G,\mu}$. Moreover, because the Igusa stack is insensitive to quasi-isogenies, we expect that $\Igs_{G,\mu}$ depends only on the Galois average of $\mu$.

Assume that $(G,\mu)$ is a Shimura datum and that we have an embedding $(G,\mu)\to (\GSP,\mu)$, where $\mu$ is not a minuscule cocharacter for $\GSP$. However, assume there exists another cocharacter $\mu^\triangleright$ of $\GSP$ such that $(\GSP,\mu^\triangleright)$ is a Shimura datum and that $\mu$ and $\mu^\triangleright$ have the same Galois average. We define the Igusa stack for $(G,\mu)$ as a substack of $\Igs_{\GSP,\mu^\triangleright}$ consisting of points whose associated $\ell$-adic local systems preserve suitable tensors. Using the Cartesian square mentioned above, we then form a candidate for the integral model. To validate this construction, we must prove that the Igusa stack is non-empty. Once we show that the Igusa stack contains enough elliptic points, we can apply Varshavsky's characterization of Shimura varieties to prove that the generic fiber of this candidate is, in fact, the desired Shimura variety.

In this paper, we establish Scholze’s fiber product conjecture for meta-unitary Shimura varieties. Since these Shimura varieties are already proper, no compactification is required in our setting. We construct the associated Igusa stack and verify that it satisfies all the properties predicted by Scholze’s conjecture.

\begin{thm}\label{fiberproductconjecture}
Let $(G,\mu)$ be a meta-unitary Shimura datum, and let $K^p \subset G(\mathbb{A}_f^{\,p})$ be a compact open subgroup away from $p$. There exists an Artin $v$-stack $\Igs_{K^p,G,\mu}$ which, for every prime $\ell \neq p$, is $\ell$-cohomologically smooth of dimension $0$, and whose dualizing complex of $\Igs_{K^p,G,\mu}(G,\mu) \times \Spd(\mathbb{F}_p)$ is isomorphic to $\mathbb{F}_p[0]$. Moreover, this stack satisfies all the properties stated in Conjecture \ref{mainconjecture}.
\end{thm}
Using the results of \cite{kim2025uniqueness} concerning the functoriality of the Igusa stack, we are able to extend our construction to all Shimura data \((G, \mu)\) that admit an embedding into a meta-unitary Shimura datum. It is plausible that this approach could be further exploited to establish the existence of the Igusa stacks for all Shimura data of type \(E_7\).

\subsection{Some Cohomological Consequences}

Using the Igusa stack constructed above, one can define a complex of sheaves $\mathscr{F} \in D(\Bun_G, \Lambda)$ that \emph{controls} the compactly supported \'etale cohomology of the associated Shimura varieties. More precisely, for a torsion $\mathbb{Z}_\ell$-algebra $\Lambda$, we set
\[
\mathscr{F} := j_{!} R(\pi_{\overline{\mathrm{HT}},\, \overline{\mathbb{F}}_p})_{\ast} \Lambda,
\]
where $j \colon \Bun_{G,\mu^{-1}} \hookrightarrow \Bun_G$ denotes the natural open immersion.

Let $(G,\mu)$ be a meta-unitary Shimura datum, and let $E$ be its reflex field. Denote by $W_E$ the Weil group of $E$. Consider the Hecke operator
\[
T_{\mu} \colon D(\Bun_G, \mathbb{F}_p, \Lambda) \longrightarrow D(\Bun_G, \mathbb{F}_p, \Lambda)^{B W_E},
\]
and let $i_1 \colon \Bun_G^1 \hookrightarrow \Bun_G$ be the natural inclusion. In this setting we prove:
\begin{thm}\label{controlingsheaf}
Let $(G,\mu)$ be a meta-unitary Shimura datum, and let $\Lambda$ be a torsion $\mathbb{Z}_{\ell}$-algebra containing a square root of $p$. Then there exists a canonical $G(\mathbb{Q}_p) \times W_E$-equivariant isomorphism
\[
i_{1, \mathbb{F}_p}^{\ast} \, T_{\mu, \mathbb{F}_p, !} \mathscr{F}
\;\simeq\;
R\Gamma_{\mathrm{\acute{e}t}}\bigl(\mathrm{Sh}_{K^p}(G,\mu)_E, \Lambda\bigr).
\]
\end{thm}
We then apply Theorem~\ref{controlingsheaf} to establish the compatibility between the Fargues--Scholze local Langlands correspondence and the \'etale cohomology of meta-unitary Shimura varieties. Fix a compact open subgroup $K_p \subset G(\mathbb{Q}_p)$, and set $K = K_p K^p$. Consider the \'etale cohomology complex
\[
R\Gamma_{\textup{\'et}}(\Sh_K(G,\mu)_E, \Lambda),
\]
which carries a natural right action of the local unramified Hecke algebra $\mathcal{H}_{K_p}(G(\mathbb{Q}_p), \Lambda)$. Denote by $\mathcal{Z}_{K_p}$ the center of $\mathcal{H}_{K_p}$. For a character $\chi : \mathcal{Z}_{K_p} \to L$, define the associated $W_E$-representation
\[
W^i(\chi) := H^i_{\textup{\'et}}(\Sh_K(G,\mu)_E, \Lambda) \otimes_{\mathcal{Z}_{K_p}, \chi} L.
\]

\begin{thm}\label{compatabilitywithScholzefargueintro}
Let $(G,\mu)$ be a meta-unitary Shimura datum, and assume that the order of $\pi_0(Z(G))$ is prime to $\ell$. Then every irreducible $L$-linear representation of $W_E$ appearing as a subquotient of $W^i(\chi)$ also appears as a subquotient of $(r_\mu \circ \varphi_\chi)|_{W_E}$.
\end{thm} 
We will also prove the Mantovan formula and its geometric incarnation in our setting:

\begin{cor}\label{mantovanformulegeometricintro}
The integral Shimura variety admits a Newton stratification by $\CS_K^b$, and we have
\[ \CS_K^b = [\Ig_b/\tilde{G}_b] \times_{[*/\tilde{G}_b]} [\CM^{\textup{int}}_{G,\mu,b}/\tilde{G}_b]. \]
\end{cor}

\begin{thm}\label{mantovan formulecohomologicintro}
There exists a filtration on the complex of smooth representations $\textup{R}\Gamma(\Sh_{K_p, E}, \Lambda)$, indexed by the set of Newton strata $[b] \in B(G)$, whose graded pieces are given by
\[ \textup{gr}_{[b]} \textup{R}\Gamma(\Sh_{K_p, E}, \Lambda) \simeq \textup{R}\Gamma(\CI^{b,v}_{K_p}, \Lambda)^{\textup{op}} \otimes_{\CH(G_b)}^L \textup{R}\Gamma_c(\CM_{G,b,\mu,\infty}, \delta_b)[2d_b]. \]
\end{thm}

Next we observe that the proof of \cite{yang2025generic} works in the setting of meta-unitary Shimura varieties. Let $\ell \neq p$ be a prime, and let $K_p \subset G(\mathbb{Q}_p)$ be a hyperspecial subgroup. Denote by $\mathcal{H}_{K_p}$ the $\mathbb{F}_\ell$-valued spherical Hecke algebra for $G(\mathbb{Q}_p)$ with respect to $K_p$. Let $\overline{E}$ be an algebraic closure of the reflex field $E$. Then $\mathcal{H}_{K_p}$ acts on the \'etale cohomology groups
\[
H^i_{\textup{\'et}}(\Sh_K(G,\mu)_{\overline{E}}, \mathbb{F}_\ell).
\]
for all compact open subgroups $K^p \subset G(\mathbb{A}_f^p)$. For a maximal ideal $\mathfrak{m} \subset \mathcal{H}_{K_p}$, let $\varphi_{\mathfrak{m}}$ denote the associated semisimple $L$-parameter.

\begin{thm}\label{torisonvanishing}
Let $(G,\mu)$ be a meta-unitary Shimura datum, and assume that $p$ is a big enough prime number. Suppose further that $\ell$ is coprime to the pro-order of $K_p$ and $\Lambda$ is either $\BF_l$ or $\BQ_l$. If the parameter $\xi_{\mathfrak{m}}$ is generic, then
\[
H^i_{\textup{\'et}}(\Sh_K(G,\mu)_{\overline{E}}, \lambda)_{\mathfrak{m}} = 0
\quad \text{for all } i \neq d,
\]
where $d = \dim \Sh_K(G,\mu)$.
\end{thm}

\subsection*{Structure of the paper}
This paper is organized as follows. In Section 2, we review the necessary background from $p$-adic geometry, including the Fargues--Fontaine curve, the moduli spaces of Shtukas, and the stack of local Shimura varieties. In Section 3, we recall the definition of meta-unitary Shimura varieties and their integral models as constructed by B\"ultel.

In Section 4, we introduce the local Igusa stack associated to a Newton stratum and study its representability. Section 5 is devoted to the construction of the global Igusa stack and the proof of the fiber product conjecture for meta-unitary Shimura varieties.

In Section 6, we derive several cohomological consequences of the fiber product formula, including the geometric Mantovan formula and the computation of the cohomology of Shimura varieties in terms of the cohomology of the Igusa stack. In Section 7, we establish the compatibility between the cohomology of meta-unitary Shimura varieties and the Fargues--Scholze local Langlands correspondence.

In Section 8, we discuss the functoriality of the Igusa stack. Finally, in Section 9, we apply the unipotent categorical local Langlands correspondence to prove the vanishing of the generic part of the cohomology.

\subsection*{Acknowledgments}
I am grateful to Mohammad Hadi Hedayatzadeh, Miaofen Chen, Shayan Gholami, Kiran Kedlaya, and Adib Abdollahi for their helpful discussions and encouragement.

\subsection*{Notation}
We follow the standard notation in $p$-adic geometry and the theory of Shimura varieties.

\begin{itemize}
    \item $p$: A fixed prime number.
    \item $E$: The reflex field of the Shimura datum $(G, \mu)$.
    \item $\mathcal{O}_E$: The ring of integers of $E$.
    \item $\breve{E}$: The completion of the maximal unramified extension of $E$.
    \item $k$ or $\mathbb{F}_q$: The residue field of $E$.
    \item $\sigma$: The Frobenius automorphism on $\breve{E}$ or $W(k)$.
    \item $G$: A reductive group over $\mathbb{Q}_p$.
    \item $\mu$: A cocharacter of $G$.
    \item $B(G)$: The set of $\sigma$-conjugacy classes of $G(\breve{\mathbb{Q}}_p)$.
    \item $\Sh_K(G, \mu)$: The Shimura variety associated to $(G, \mu)$ at level $K$.
    \item $\mathcal{S}_K$: The integral model of the Shimura variety.
    \item $\Bun_G$: The stack of $G$-bundles on the Fargues--Fontaine curve.
    \item $\Gr_G$: The affine Grassmannian of $G$.
    \item $\Sht_{G, \mu}$: The moduli stack of $(G, \mu)$-Shtukas.
    \item $\Igs_{K^p, G, \mu}$: The Igusa stack.
    \item $\Isoc_G$: The stack of $G$-isocrystals.
    \item $\Hck_G$: The Hecke stack.
    \item $X_S$: The Fargues--Fontaine curve associated with a perfectoid space $S$.
    \item $\mathcal{Y}_S$: The relative Fargues--Fontaine curve (adic space).
    \item $\Lambda$ is either $\bar{F_\ell}$ or $\bar{\BQ}_\ell$ for a big enough prime $\ell$.
    \item $D(\Bun_G, \Lambda)$: The derived category of $\Lambda$-sheaves on $\Bun_G$.
    \item ${}^L G$: The Langlands dual group of $G$.
    \item $W_E$: The Weil group of $E$.
\end{itemize}

\section{Elements from $p$-adic geometry}
In this section we want to review some concepts from $p$-adic geometry that we are going to need in the paper.


\subsection{The Fargue-Fontaine Curve}

Let $E$ be a non-archimedean local field with ring of integers $\mathcal{O}_E$ and residue field $\mathbb{F}_q$. Let $S = \Spa(R, R^+)$ be an affinoid perfectoid space over $\mathbb{F}_q$.

\begin{dfn}\label{dfn:adic_space_YS}
The adic space $\mathcal{Y}_S$ over $\mathcal{O}_E$ is defined as the complement of the vanishing locus of a pseudouniformizer $[\varpi]$ within the space of ramified Witt vectors:
\[ \mathcal{Y}_S := \Spa W_{\mathcal{O}_E}(R^+) \setminus V([\varpi]) \]
This space is equipped with a canonical Frobenius automorphism $\phi$. The generic fiber of $\mathcal{Y}_S$ over $E$ is denoted $Y_S$:
\[ Y_S := \mathcal{Y}_S \times_{\Spa \mathcal{O}_E} \Spa E \]
The action of the Frobenius $\sigma$ on $Y_S$ is free and totally discontinuous.
\end{dfn}
\begin{dfn}\label{dfn:FF_curve}
The \emph{Fargues--Fontaine curve} associated with a perfectoid space $S$ is the quotient adic space
\[
X_S := Y_S / \sigma^{\mathbb{Z}},
\]
where $Y_S$ is the relative period space and $\varphi$ denotes the Frobenius automorphism.
\end{dfn}

\begin{rmk}\label{rmk:radius_map}
Let $\kappa \colon \CY \to \mathbb{R}$ be the radius map as defined in~\cite{scholze2020berkeley}, and denote by $\CY_{[a,b]}$ the inverse image of the interval $[a,b]$ under this map. For a vector bundle $\CE$ on $\CY$, we write $\CE_{[a,b]}$ for its restriction to $\CY_{[a,b]}$.
\end{rmk}
\begin{dfn}\label{dfn:divisors}
The stack of effective Cartier divisors of degree~$1$ on the curve $\mathcal{Y}_S$, denoted $\Div_{\mathcal{Y}}^1$, is the functor on perfectoid spaces which assigns to a perfectoid space $S$ the groupoid of pairs $(S^\sharp, \tau)$, where $S^\sharp$ is an untilt of $S$ and $\tau : (S^\sharp)^{\flat} \xrightarrow{\;\sim\;} S$ is an isomorphism. 

The stack of effective Cartier divisors of degree~$1$ on the Fargues--Fontaine curve, denoted $\Div^1$, is defined similarly, except that we identify two such pairs if their isomorphisms differ by a power of the Frobenius.
\end{dfn}

\begin{dfn}\label{dfn:divisors_r}
For a positive integer $r$, the stacks of effective Cartier divisors of degree~$r$ on $\mathcal{Y}_S$ and on the Fargues--Fontaine curve are denoted by $\Div_{\mathcal{Y}}^r$ and $\Div^r$, respectively. They are defined as the quotient stacks
\[
\Div_{\mathcal{Y}}^r := [(\Div_{\mathcal{Y}}^1)^r / \Sigma_r], \qquad
\Div^r := [(\Div^1)^r / \Sigma_r],
\]
where $\Sigma_r$ is the symmetric group on $r$ elements acting by permutation of the components.
\end{dfn}

\begin{rmk}\label{rmk:period_rings}
Let $S$ be a perfectoid space. Then a point $x \in \Div_{\CY}^r(S)$ determines a closed adic subspace $\Gamma_x \subset \mathcal{Y}_S$,
corresponding to the effective Cartier divisor on the relative curve associated with $x$.
Let $I_x \subset \mathcal{O}_{\mathcal{Y}_S}$ denote the ideal sheaf defining the closed subspace $\Gamma_x$. We then define the associated period rings
\[
B_{x}^+ := \widehat{\mathcal{O}}_{\mathcal{Y}_S, I_x}, \qquad B_x := B_{x}^+[\frac{1}{I_x}],
\]
\end{rmk}

\subsection{Shtukas}

Let $G$ be a reductive group over $\BQ_p$ and $\mu$ be a cocharacter of $G$ defined over the local field $E$.

\begin{dfn}\label{dfn:Bun_G}
    The \emph{moduli stack of $G$-bundles}, denoted $\Bun_G$, is the v-stack whose $S$-points classify $G$-bundles on the relative Fargues--Fontaine curve $X_S$. If $S$ is a geometric point, the set of isomorphism classes of $G$-bundles on $X_S$ is in bijection with the Kottwitz set $B(G)$. This induces a stratification $$\Bun_G = \coprod_{b \in B(G)} \Bun_G^b$$, where each $\Bun_G^b$ is a locally closed substack corresponding to bundles of isocrystal class $b$. Associated to each $b \in G(\breve{\mathbb{Q}}_p)$ is a connected reductive group $G_b$ over $\mathbb{Q}_p$, the $\sigma$-centralizer of $b$. There exists an extension $\tilde{G}_b$ of $G_b$ by a unipotent group such that $$\Bun_G^b \simeq [*/\tilde{G}_b]$$.
    The pullback along this map induces an equivalance 
    $$D(\Bun_G^b.\Lambda)\simeq D([*/G^b(E)])$$
\end{dfn}

\begin{dfn}\label{dfn:modification}
Let $x \subset Y_S$ be an effective Cartier divisor. Denote by $U = Y_S \setminus D$ the open complement. Let $\CE_1$ and $\CE_2$ be vector bundles on $Y_S$. A \emph{modification of $\CE_1$ into $\CE_2$ along $x$} is an isomorphism
\[
\phi \colon \CE_1|_U \xrightarrow{\;\sim\;} \CE_2|_U.
\]
Two such modifications $(\CE_2, \phi)$ and $(\CE_2', \phi')$ are said to be equivalent 
if there exists an isomorphism $\psi \colon \CE_2 \xrightarrow{\sim} \CE_2'$ such that 
$\phi' = \psi|_U \circ \phi$.

A \emph{datum of modification} of $\CE_1$ at a Cartier divisor $x$ consists of a vector bundle 
 $\hat\CE_{2,x}$ over $B^+_{x}$ together with an isomorphism $\eta:\hat\CE_{2,x}[\frac{1}{I_x}]\to \CE_1\otimes B_{x}$.   
\end{dfn}
\begin{rmk}\label{rmk:BL_theorem}
By the Beauville–Laszlo theorem, giving a datum of modification of $\CE_1$ at $x$ is equivalent to giving a vector bundle $\CE_2$ on $Y_S$ together with a modification of $\CE_1$ into $\CE_2$ along $x$ that agrees with the given datum after base change to $B^+_{x}$.
\end{rmk}

\begin{dfn}\label{dfn:BL_morphism}
The \emph{Beauville-Laszlo morphism} provides a uniformization of $\Bun_G$ by the deRham affine Grassmannian, $\Gr_G$. It is a surjective map of v-stacks:
\[ \mathop{BL}: \Gr_G \to \Bun_G \] 
\end{dfn}

\begin{dfn}\label{dfn:shtukas}
Let $G$ be a reductive group over $\BQ_p$ with smooth model $\CG$ over $\BZ_p$, let $I$ be a finite index set of size $r$, and let $\mu = (\mu_i)_{i \in I}$ be a collection of dominant minuscule cocharacters of $\CG$ defined over 
\[
\CO_{E_I} = \prod_{i\in I} \CO_{E_i}.
\]

The \emph{moduli space of shtukas}, denoted $\Sht_{(G,\mu)}^{I}$, is the functor defined over $\Perf/\Spd \CO_{E_I}$,
which sends a perfectoid space $S$ to the groupoid of quintuples $(x,\CE,\iota)$ consisting of:
\begin{enumerate}
    \item a morphism $x \in \Div_{\CY}^r$ defining a collection of Cartier divisors 
    \(
        \Gamma_{x_i} \subset \CY_S,
    \)
    \item a $\CG$-bundle $\CE$ on the curve $\CY_S$,
    \item an isomorphism of $\CG$-bundles
    \[
    \iota \colon 
      \sigma^*\CE\big|_{\CY_S \setminus \Gamma_x} 
      \xrightarrow{\ \sim\ } 
      \CE\big|_{\CY_S \setminus \Gamma_x},
    \]
    bounded by~$\mu$.
\end{enumerate}
\end{dfn}

\begin{rmk}\label{rmk:shtuka_bundle}
A Shtuka with no leg is equivalent to a $G$-bundle on the Fargues–Fontaine curve $X$. There is an equivalence between the points of the stack $\Bun_G$ and the Kottwitz set $B(G)$. For each $b \in B(G)$, we denote by $\CE_b$ the corresponding Shtuka with no legs. 

Let $\epsilon < \min_{i \in I} \kappa(x_i)$ and $r > \max_{i \in I} \kappa(x_i)$. Using the fact that the Frobenius is a contraction on the curve $Y_S$, we can extend $\CE_{[0,\epsilon]}$ and $\CE_{[r,\infty]}$ to Shtukas with no legs, denoted by $\CE_0$ and $\CE_\infty$, respectively. 

By the Beauville–Laszlo theorem, having a Shtuka with legs at the points $x_1, x_2, \dots, x_n$ is equivalent to specifying a Shtuka with no legs $\CE_0$ together with the data of its modifications at $x_1, x_2, \dots, x_n$.
\end{rmk}
\begin{rmk}\label{rmk:shtuka_locsys}
A shtuka with no leg, $\CE_0$, over a perfectoid space $S$ is equivalent to an \'etale $\BQ_p$-local system, $T_0$, over $S$. For a shtuka with one leg, $\CE$, let $\CE_\infty$ be the associated shtuka with no leg defined above and $T_\infty$ be the corresponding local system. We denote this association by $\pi_{\textup{\'et}}(\CE) = T_\infty$.
\end{rmk}
We also require a local version of the moduli space of Shtukas that lies over a given Newton stratum:
\begin{dfn}\label{dfn:integral_local_shimura}
Let $(G, b, \mu)$ be an integral Shimura datum, consisting of a smooth group scheme $\CG$ over $\BZ_p$, a cocharacter $\mu$ of $\CG$ defined over some extension of $\BZ_p$, and an element $b$ in the Kottwitz set $B(G,\mu) \subset B(G)$.

The \emph{integral local Shimura variety} associated with this datum, denoted
\(
\CM^{\mathrm{int}}_{G,\mu},
\)
is the functor defined over $\Perf / \Spd O_E$, 
which sends a perfectoid space $S$ to the set of equivalence classes of triples $(x, \CE, \tau_r)$, where:
\begin{enumerate}
    \item $x \in \Div_{\CY}^r$ is a morphism defining a collection of Cartier divisors 
    \(
       \Gamma_{x_i} \subset \CY_S.
    \)
    \item $\CE$ is a $\CG$-Shtuka bounded by $\mu$ (more precisely, the corresponding modification factors through the subsheaf $\CM^{\mathrm{loc}}_{\CG,\mu}$ defined in \cite[Chapter~25]{scholze2020berkeley}).
    \item $\tau_r$ is an isomorphism
    \[
      \tau_r : \CE_{[r,\infty]} \xrightarrow{\ \sim\ } \CE_{b,[r,\infty]},
    \]
    where $\CE_b$ denotes the Shtuka associated with $b$.
\end{enumerate}
\end{dfn}

\begin{rmk}\label{rmk:local_shimura_generic}
The \emph{local Shimura variety} is the generic fiber of the integral local Shimura variety, we denote it by $\mathcal{M}_{G,\mu}$.
\end{rmk}
Over the generic fiber, one can give an equivalent definition in the terms of the vector bundles on the Fargues-Fontaine curve, one can also add the level structure:
\begin{dfn}\label{dfn:local_shimura_level}
Let $K\subset G(\BQ_p)$ be a compact open subgroup, the local Shimura variety associated to the Shimura datum $(G, \mu, b)$ is the functor from the category of perfectoid spaces over $\Spd E$, denoted $\Perf_{/\Spd E}$, to the category of groupoids. This functor assigns to a space $S \in \Perf_{/\Spd E}$ the groupoid of quadruples $(x, \mathcal{E}, \iota, \mathbb{P})$, where:
\begin{enumerate}
    \item A map $x: S \to \operatorname{Div}_X^1$ that defines a relative Cartier divisor $\Gamma_x \subset X_S$ of degree one.
    \item A $G$-bundle $\mathcal{E}$ on the relative Fargue-Fontaine curve $X_S$ that becomes trivial over any geometric point of $S$.
    \item An isomorphism of $G$-bundles
    \[ \iota: \mathcal{E}|_{X_S \setminus \Gamma_x} \xrightarrow{\sim} \mathcal{E}^b|_{X_S \setminus \Gamma_x} \]
    which is meromorphic along the divisor $\Gamma_x$ and is bounded by $\mu$.
    \item A $K$-lattice $\mathcal{T}$ within the \'etale local system $\pi_{\textup{\'et}}(\mathcal{E})$.
\end{enumerate}
\end{dfn}

\begin{thm}\label{thm:shtuka_representability}
The functor $\Sht_{(G,\mu,K)}^{I}$ is representable by a spatial diamond, this diamond is a v-stack which is separated and locally of finite type over $\Spd(E)$.
\end{thm}

\begin{dfn}\label{dfn:hecke_stack}
Let $\CY$ denote the relative Fargues--Fontaine curve and $\CG$ a reductive group over $\BZ_p$.  
The \emph{Hecke stack} $\Hck_G$ classifies modifications of $G$-bundles on $\CY$ at effective Cartier divisors.

For a perfectoid space $S$ over $\Spd E$, the groupoid $\Hck_G(S)$ consists of tuples $(\CE_1, \CE_2, x, \phi)$,
where
\begin{itemize}
    \item $\CE_1$ and $\CE_2$ are $G$-bundles on $\CY_S$,
    \item $x \in \Div_{\CY}(S)$ is an effective Cartier divisor on $\CY_S$, and
    \item $\phi \colon \CE_1|_{\CY_S \setminus \Gamma_x} \xrightarrow{\;\sim\;} \CE_2|_{\CY_S \setminus \Gamma_x}$
          is an isomorphism of $G$-bundles on the complement of the graph $\Gamma_x \subset \CY_S$.
\end{itemize}

This stack naturally fits into the \emph{Hecke correspondence}:
\[
\begin{tikzcd}[row sep=large, column sep=huge]
 & \Hck_G \arrow[dl, "h_1"'] \arrow[dr, "{(h_2,\;x)}"] & \\
 \Bun_G & & \Bun_G \times \Div_{\CY}
\end{tikzcd}
\]
where
\[
h_1(\CE_1, \CE_2, x, \phi) = \CE_1, \qquad
h_2(\CE_1, \CE_2, x, \phi) = \CE_2, \qquad
x(\CE_1, \CE_2, x, \phi) = x.
\]
\end{dfn}

\begin{dfn}\label{dfn:bounded_hecke}
Fix a cocharacter $\mu$ of $G$ defined over a suitable extension of $\BZ_p$.  
The \emph{bounded Hecke stack} $\Hck_{G, \le \mu} \subset \Hck_G$ is the substack parameterizing only those modifications whose relative position at $x$ is bounded by $\mu$. 
\end{dfn}

\begin{dfn}\label{dfn:period_map}
The \emph{period map} is a morphism of v-stacks that forgets the shtuka isomorphism and level structure:
\[ \pi: \Sht_{G,\mu,K} \to \Bun_{G,\mu^{-1}} \]
That sends a shtuka $\CE$ to the $G$-bundle associated to the shtuka with no leg $\CE_0$.
\end{dfn}

\begin{dfn}\label{dfn:shtuka_over_F}
Let $\CG$ be a reductive group over $\mathbb{Z}_p$, and let $F$ be a $v$-sheaf on the category of perfectoid spaces over $\Spd(\mathbb{Z}_p)$.  

A \emph{$\CG$-shtuka over $F$} is a rule which assigns to every morphism 
\[
x : \Spd(R) \longrightarrow F
\]
from an affinoid perfectoid $\operatorname{Spd}(R)$ a $G$-shtuka $\mathcal{E}_x$ over $R$, in a way that is compatible with pullback. 

A $G$-shtuka on $F$ determines a pro-\'etale $\CG(\mathbb{Z}_p)$-local system on the \'etale site $F_{\et}$. If the adic generic fiber $F^{\mathrm{rig}}$ of $F$ is representable by a rigid-analytic space, this further induces a pro-\'etale $G(\mathbb{Z}_p)$-local system on $F^{\mathrm{rig}}$.  

In particular, when $S$ is a smooth rigid-analytic space over $\mathbb{Q}_p$, this construction yields an equivalence between de~Rham $\CG(\mathbb{Z}_p)$-local systems on $S$ and $\CG$-shtukas on~$S$.
\end{dfn}

\subsection{Shimura varieties and the Hodge tate period map}
Let $(G, \mu)$ be a Shimura datum, and let $K = K^p K_p \subset G(\mathbb{A}_{\mathbb{Q}}^{\infty})$ be a compact open subgroup. 
Consider the corresponding Shimura variety $\Sh(G, \mu)_K$.  
We denote by $\Sh^{\mathrm{rig}}(G, \mu)_K$
the analytification of this Shimura variety as a rigid-analytic space, and by $\mathcal{S}_K$
its integral model (in the cases where such an integral model exists).  

We define the Shimura variety at the \emph{$p^\infty$-level} as the inverse limit over all open compact subgroups $K_p \subset G(\mathbb{Q}_p)$:
\[
\Sh_{K^p} := \varprojlim_{K_p} (\Sh_{K_p K^p}(G,\mu))^\diamondsuit.
\]
The pro-\'etale cover $\mathcal{S}_{K^p}$ induces a $K_p$-local system 
$\mathbb{V}_{\mathrm{\acute{e}t}}$
on $\Sh_K(G,\mu)$.  
By \cite{liu2017rigidity}, this local system is de~Rham, and therefore gives rise to a morphism
\[
\pi_{\mathrm{HT}} : \Sh_{K^p} \longrightarrow \Gr_{G, \mu^{-1}}.
\]
It also determines a $G$-shtuka over $\Sh(G,\mu)$. This gives us a period map
\[
\pi_{\mathrm{Crys}}:\Sh_K(G,\mu)\to \Sht_{G,\mu}
\]
For further details on this construction, we refer the reader to \cite{pappas2022p}.

\subsection{Stack of ($G,\mu$)-displays}\label{sectiondisplays}
Let $\CG$ be a reductive group scheme over $\BZ_p$ and let $\mu$ be a cocharacter of $\CG$ defined over an extension of $\BZ_p$. In \cite{hedayatzadeh2024deformations}, we defined the stacks $\CB^{\perf}_{\CG,\mu}$ and $\CB^{\perf}_{\CG,\mu,b}$. We now briefly recall the construction.

Let \( R^+ \) be an integral perfectoid ring, and consider the étale site \( (R^+)_{\mathrm{\acute{e}t}} \) of integral perfectoid \( R^+ \)-algebras. Let
\[
\theta \colon A_{\mathrm{inf}}(R^+) \to R^+
\]
be the canonical surjection, and let \( (d) := \ker(\theta) \). For each \( i \in \mathbb{Z} \), the Nygaard filtration is defined by
\[
N^i A_{\mathrm{inf}}(R^+) := \sigma^{-i}(d),
\]
where \( \sigma \) denotes the Frobenius endomorphism.

Let \( \CG = \operatorname{Spec} S \) be a reductive group scheme over \( \mathbb{Z}_p \). The cocharacter \( \mu \) induces a filtration on \( S \).

\begin{dfn}\label{loopgroupdfn}
We define sheaves of groups \( L^+\CG \) and \( \CG^\mu \) on the site \( (R^+)_{\mathrm{\acute{e}t}} \) as follows. For an object \( S^+ \) of \( (R^+)_{\mathrm{\acute{e}t}} \),
\[
L^+\CG(S^+) := \CG\bigl(A_{\mathrm{inf}}(S^+)\bigr),
\]
and \( \CG^\mu(S^+) \) is the group of morphisms \( \operatorname{Spec}(S^+) \to \CG \) that are compatible with the filtrations induced by \( \mu \) and the Nygaard filtration.
\end{dfn}

Consider the morphism of sheaves
\[
\Phi_d^\mu \colon \CG^\mu \to L^+\CG,
\]
defined on sections by
\[
\Phi_d^\mu(k) = \mu(d)\,\sigma(k)\,\mu(d)^{-1}.
\]
Using this morphism, we define a right action of \( \CG^\mu \) on \( L^+\CG \) by
\[
U \cdot k := k^{-1} \, U \, \Phi_d^\mu(k),
\]
for \( k \in \CG^\mu(S^+) \) and \( U \in L^+\CG(S^+) \).

\begin{dfn}
We define the stack \( \CB^{\perf}_{\CG,\mu} \) as the quotient stack
\[
\CB^{\perf}_{\CG,\mu} := [L^+\CG / \CG^\mu].
\]
\end{dfn}

\begin{rmk}\label{rmk:B_perf_embedding}
Let \(R^+\) be an integral perfectoid ring. 
By \cite[Proposition 8.3]{hedayatzadeh2024deformations}Then there exists a fully faithful embedding
\[
\CB^{\perf}_{\CG,\mu}(R^+)\;\hookrightarrow\;
\Sht_{\CG,\mu}(R^+),
\]
whose essential image consists precisely of those shtukas that can be extended to the~\(\infty\)-point of the curve~$\CY$. In particular we have an isomorphsim There is an isomorphism between the stacks 
\(\Sht_{G,\mu,b}\) and \(\CB^{\perf}_{\CG,\mu,b}\).
\end{rmk}

Let $\CG_0$ be a reductive group defined over $\BZ_{p^r}$ and $\CG=\Res_{\BZ_{p^r}/\BZ_p} \CG_0$. Let $\boldsymbol{\mu}$ be a cocharacter of $\CG_0^r$ defined over $\BZ_{p^r}$. 
Define $\Gamma\subset \BZ/r\BZ$ to be the subset consisting of those $i$ for which $\mu_i\neq \mathrm{id}$. 
Consider the natural bijection
\[
\mathrm{mod} \colon \mathbb{Z}/r\mathbb{Z} \xrightarrow{\sim} \{1,2,\dots,r\},
\]
and let $\gamma_i$ denote the $i$-th element of $\Gamma$ under the ordering induced by this bijection.  
Set
\[
r_i^{+} := \mathrm{mod}(\gamma_{i+1} - \gamma_i) \in \mathbb{N}.
\]
For each $1 \le i \le r$, define
\[
[\gamma_i, \gamma_{i+1}) \subset \mathbb{Z}/r\mathbb{Z}
\]
to be the subset consisting of those elements $w \in \mathbb{Z}/r\mathbb{Z}$ such that
\[
\gamma_i \le w < \gamma_{i+1} \quad \text{if } i < |\Gamma|,
\qquad 
\text{and} \qquad 
\gamma_i < w \text{ or } w < \gamma_1 \quad \text{if } i = |\Gamma|.
\]

Finally, for $w \in [\gamma_i,\gamma_{i+1})$, define
\[
t^{+}(w) := \mathrm{mod}(w - \gamma_i).
\]

Let $\pi\colon G_0\to \GL(V)$ be a representation, and for $i\in \BZ/r\BZ$, let $w_i$ be the highest weight of $\pi\circ\mu_i$. 
We call this representation \emph{sparse} if, for every $\gamma\in \Gamma$, we have $r_i^+ \ge w_i$. 
In this setting, one can define (see \cite[Construction]{hedayatzadeh2025embeddings}) the cocharacter $\boldsymbol{\mu}^\triangleright$ of $G$ such that 
\[
\prod_{\gamma_{i-1}+1}^{\gamma_i}\boldsymbol{\mu}^\triangleright_{w}=\boldsymbol\mu_{\gamma_i}
\quad\text{and}\quad
\sum_{\sigma\in \Gal(\BQ_{p^r}/\BQ_p)} \boldsymbol\mu^\triangleright=\sum_{\sigma\in \Gal(\BQ_{p^r}/\BQ_p)}\boldsymbol\mu.
\]
Let $\tilde{\CG} := \Res_{\BZ_{p^r}/\BZ_p}\GL(V)$. 
Assume that $\pi$ is a sparse representation. 
We call the datum $(\CG,\mu,\pi)$ a \emph{sparse display datum}.

Let $(\CG,\mu,\pi)$ be a sparse Shimura datum. 
There exists a functor 
\[
\Flex:\CB^{\perf}_{\CG,\mu}\longrightarrow\CB^{\perf}_{\tilde{\CG},\mu^\triangleright}
\]
constructed in~\cite[Construction 3.19]{hedayatzadeh2025embeddings}.

\begin{pro}[{\cite[Prop.~3.16]{hedayatzadeh2025embeddings}}]\label{pro:flex_isogeny}
For any perfectoid space~$S$ and $D\in\CB^{\perf}_{\CG,\mu}(S)$,
there exists an isogeny modulo~$p$ between $D$ and $\Flex(D)$.
\end{pro}

When the underlying Shimura datum is regularly sparse, 
the functor~$\Flex$ induces a closed immersion between the corresponding local Shimura varieties.

\begin{thm}[{\cite[Lem.~5.4]{hedayatzadeh2025embeddings}}]\label{thm:flex_immersion}
Let $(G,\mu,b)$ be a regularly sparse integral Shimura datum.
Then the morphism $\Flex\colon \CM^{\mathrm{int}}_{G,\mu,b} \longrightarrow \CM^{\mathrm{int}}_{\tilde{G},\,\mu^\triangleright,b}$
between the associated integral local Shimura varieties is representable by a closed immersion.
\end{thm}
\begin{rmk}\label{rmk:upgrade_flex}
Over the generic fiber, one can upgrade $\Flex$ to an embedding between the local Shimura varieties with deeper levels.
\end{rmk}

\subsection{The Stack of Isocrystals and Local Shtukas}

Let $E$ be an unramified extension of $\BQ_p$ with ring of integers $\CO_E$ and residue field $k$. Let $\CG$ be a smooth group defined over $\CO_E$ with the generic fiber $G$. In what follows, we work within the category $\Perf_k$ of perfect affine $k$-algebras.

We denote the Witt-vector loop group $L^{\textup{witt}}G$ and the positive loop group $L^{+,\textup{witt}}\CG$ as the functors sending a perfect $k$-algebra $R$ to $G(W(R)[1/p])$ and $\CG(W(R))$, respectively. Let $\sigma$ denote the Frobenius automorphism acting on the Witt vectors.

\begin{dfn}\label{dfn:Isoc_G}
    The \emph{stack of $G$-isocrystals}, denoted by $\Isoc_G$, is defined as the étale quotient stack:
    \[
    \Isoc_G := L^{\textup{witt}}G / \Ad_\sigma L^{\textup{witt}}G,
    \]
    where $\Ad_\sigma$ denotes the $\sigma$-twisted conjugation action. As a functor of points, for $R \in \Perf_k$, the groupoid $\Isoc_G(R)$ classifies pairs $(\mathcal{E}, \phi_{\mathcal{E}})$, where $\mathcal{E}$ is a $G$-torsor on $\Spec W(R)[\frac{1}{p}]$ and $\phi_{\mathcal{E}}: \sigma^* \mathcal{E} \xrightarrow{\sim} \mathcal{E}$ is an isomorphism.
\end{dfn}

By \cite{gleason2023meromorphic} there exists an isomorphism between the underlying reduced stack of $\Bun_G$ and the stack of isocrystals:
\[
\Bun_G^{\mathrm{red}} \xrightarrow{\sim} \Isoc_G.
\]

We now turn to the integral structure.

\begin{dfn}\label{dfn:Sht_loc}
    The \emph{moduli stack of local shtukas} (in the Witt vector formulation), denoted by $\Sht_\CG^{\mathrm{loc}}$, is defined as the étale quotient stack:
    \[
    \Sht_\CG^{\mathrm{loc}} := L^{\textup{witt}}G / \Ad_\sigma L^{+,\textup{witt}}\CG.
    \]
    The groupoid of points classifies pairs $(\mathcal{E}, \phi_{\mathcal{E}})$ where $\mathcal{E}$ is a $\CG$-torsor on $\Spec(W(R))$ and $\phi_{\mathcal{E}}$ is a meromorphic isomorphism of the associated $G$-torsors over $\Spec(W(R)[\frac{1}{p}])$.
\end{dfn}

By \cite[Lemma 3.1.7]{daniels2024Igusa}, there is an isomorphism between the reduced locus of the shtuka stack and the Witt vector stack defined above:
\[
\Sht_{\CG}^{\mathrm{red}} \xrightarrow{\sim} \Sht_\CG^{\mathrm{loc}}.
\]

There is a natural morphism connecting the integral and rational structures, known as the Newton morphism:
\[
\Nt: \Sht_\CG^{\mathrm{loc}} \longrightarrow \Isoc_\CG.
\]

The points of the stack $\Isoc_G$ are classified by the set of $\sigma$-conjugacy classes $B(G)$. For each element $b \in B(G)$, we have the corresponding Newton stratum, realized as a locally closed substack immersion:
\[
i_b: \Isoc_{G,b} \hookrightarrow \Isoc_G.
\]
Furthermore, the category of sheaves on a specific Newton stratum admits a representation-theoretic description. There is an equivalence:
\[
\Shv(\Isoc_{G,b}) \simeq \Rep(G_b(\BQ_p), \Lambda),
\]
where $G_b$ is the $\sigma$-centralizer group of $b$, and $\Lambda$ is the coefficient ring.

 \begin{pro}\label{existanceoflattice}\cite[Proposition 3.2.3]{daniels2024Igusa}
The map $\Sht_{G,\mu}\to \Isoc_{G,\mu^{-1}}$ is v-surjective.
\end{pro}

\subsection{Fargues-Scholze local langlands}
\begin{dfn}\label{dfn:LocSys}
Let $E$ be a non-archimedean local field with Weil group $W_E$, and let $\Ghat$ be the Langlands dual group of a reductive group $G$ over $E$. The \textbf{stack of L-parameters}, or \(\Ghat\)-local systems, is the quotient stack
\[
\mathrm{LocSys}_{\Ghat} := [Z^1(W_E, \Ghat) / \Ghat]
\]
where $Z^1(W_E, \Ghat)$ is the space of continuous 1-cocycles (L-parameters) $\phi: W_E \to \Ghat$, and $\Ghat$ acts by conjugation. This is an algebraic stack over $\BZ_\ell$.
\end{dfn} 

Fargues and Scholze defined the full subcategories $D^{\mathrm{ULA}}(\Bun_G) \subset D(\Bun_G)$ and $D^{\mathrm{ULA}}(\Hck_G) \subset D(\Hck_G)$ consisting of universally locally acyclic sheaves. They further endow the $D(\Bun_G)$ and $D(\Hck_G)$, with perverse $t$-structures.

Let $V$ be a representation of the Langlands dual group ${}^L G$. The geometric Satake equivalence produces a ULA perverse sheaf $\mathcal{S}_V$ on the Hecke stack $\Hck_G$ associated to $V$. Consider the standard correspondence diagram

\[
\begin{tikzcd}[row sep=large, column sep=huge]
 & \Hck_G \arrow[dl, "h_1"'] \arrow[dr, "{(h_2,\,x)}"] & \\
 \Bun_G & & \Bun_G \times \Div_{\CY}
\end{tikzcd}
\]

The \emph{Hecke operator} associated to $V$, denoted $T_V$, is defined as the functor
\[
T_V: D(\Bun_G) \longrightarrow D(\Bun_G)\times \Div, \quad \mathcal{F} \longmapsto Rh_{2!}(h_1^*\mathcal{F} \otimes^{\mathbb{L}} \mathcal{S}_V),
\]
One can see that this operators factor through $D(\Bun_G)^{\BBW}$, consisting of $W$-equivariant objects in $D(\Bun_G)$.Let $\mu$ be a dominant cocharacter of $G$ and let $V_\mu$ be the highest weight representation associated to $\mu$; we define the Hecke operator $T_\mu$ as $T_{V_\mu}$.

One can upgrade the Hecke operators to the spectral action:

\begin{dfn}\label{dfn:spectral_action}
Let $\Bun_G$ denote the moduli stack of $G$-bundles on the Fargues–Fontaine curve, 
and consider the derived category $D(\Bun_G, \lambda)$.
Let $\Perf(\mathrm{LocSys}_{\Ghat})$ be the derived category of perfect complexes on the moduli stack 
$\mathrm{LocSys}_{\Ghat}$ of $\Lambda$-valued Langlands parameters, as in \cite{fargues2021geometrization}.
Then there exists a monoidal $\Lambda$-linear action
\[
\Perf(\mathrm{LocSys}_{\Ghat}) \longrightarrow \End(D(\Bun_G, \Lambda))^{\BBW},
\quad 
C \longmapsto \bigl(A \mapsto C \star A\bigr),
\]
called the \emph{spectral action}.
For $V \in \Rep_{\Lambda}({}^L G)$, the induced endofunctor 
$C_V \star (-)$ coincides with the Hecke operator $T_V$ acting on $D(\Bun_G,\Lambda)$.
\end{dfn}

Fargues and Scholze constructed a semisimple Local Langlands correspondence, denoted $\mathrm{LLC}$, from the set of isomorphism classes of irreducible smooth representations of $G(E)$ to the set of equivalence classes of semisimple $L$-parameters.

\begin{dfn}\label{dfn:hecke_algebra}
Let $G$ be a reductive group over $\BQ_p$ which is unramified, and let 
$K_p \subset G(\BQ_p)$ be a compact open subgroup.  
The \emph{Hecke algebra} with coefficients in $\Lambda$ is defined as
\[
\mathcal{H}_{K_p} := \Lambda[K_p \backslash G(\BQ_p) / K_p],
\]

Given a maximal ideal $\mathfrak{m} \subset \mathcal{H}_{K_p^{\mathrm{hs}}}$, one associates to it
a Langlands parameter
\[
\varphi_{\mathfrak{m}} : W_{\BQ_p} \longrightarrow {}^L G(\Lambda),
\]
characterized by a semisimple element $\varphi_{\mathfrak{m}}(\mathrm{Frob}_{\BQ_p}) \in {}^L G(\BF_\ell)$.
\end{dfn}

When the subgroup $K_p$ is hyperspecial, the Satake isomorphism implies that the closed points of $\Spec(\mathcal{H}_{K_p})$ correspond to $\widehat{G}$-conjugacy classes of unramified semisimple $L$-parameters. Every such parameter factor through a maximal torus.

\begin{dfn}\label{dfn:generic_parameter}
Let $T$ be a maximal torus of $G$, and let $\varphi_T : W_{\BQ_p} \to {}^L T(\Lambda)$ 
be a toral Langlands parameter.  
For each coroot $\alpha \in X_*(T_{\BQ_p})/\Gamma$, consider the one-dimensional representation
$\alpha \circ \varphi_T$ of $W_{\BQ_p}$.  
We say that $\varphi_T$ is \emph{generic} if
\[
R\Gamma(W_{\BQ_p}, \alpha \circ \varphi_T) = 0
\quad\text{for all such $\alpha$.}
\]
Equivalently, $\varphi_T$ does not admit any coroot on which the Frobenius acts with eigenvalue $1$.  
A maximal ideal $\mathfrak{m} \subset \mathcal{H}_{K_p}$ is said to be \emph{generic}
if the associated toral parameter $\varphi_{T,\mathfrak{m}}$ is generic.
\end{dfn}

\section{meta-unitary Shimura variety}\label{meta-unitaryShimurasection}
B\"ultel~\cite{bultel2008pel} introduced a class of Shimura varieties known as \emph{meta-unitary Shimura varieties}. For certain level structures $K$, he constructed integral models for these varieties. In this section, we review the definition of meta-unitary Shimura varieties and their integral models.

\begin{setting}\label{genealsettingmetaunitary}
Let $L$ be a CM field with maximal totally real subfield $L^+$, and let $x \mapsto \bar{x}$ denote the generator of $\mathrm{Gal}(L/L^+)$. Let $R$ be the splitting field of $L$ and $R^+$ the splitting field of $L^+$. Denote by $L^{\mathrm{an}}$ the set of embeddings of $L$ into $\mathbb{C}$. Fix an element $\sigma \in \mathrm{Gal}(R/\mathbb{Q})$. Assume that $(V,\Psi)$ is a vector space over $L$ equipped with a skew-Hermitian form. Let $U/L^+$ be the algebraic group of $\Psi$-similarities, and set
\[
\GU/\mathbb{Q} \;=\; G_m \times \mathrm{Res}_{L^+/\mathbb{Q}} U.
\]
\end{setting}

\begin{dfn}\label{skwehermitianhodgestructure}
A \emph{skew-Hermitian rational Hodge structure of weight $-1$} on $V$ is a map
$$h : \mathrm{Res}_{\mathbb{C}/\mathbb{R}} G_m \longrightarrow \GU_{\mathbb{R}}$$
such that the symmetric form $\mathrm{tr}_{L/\mathbb{Q}} \Psi\!\left(h(i)x, y\right)$ is positive definite on $\mathbb{R} \otimes_{\mathbb{Q}} V$.
\end{dfn}

Note that $V \otimes_L \mathbb{C} = \bigoplus_{\iota \in L^{\mathrm{an}}} V_\iota$. A skew-Hermitian rational Hodge structure on $V$ determines a double grading
\[
V_{\iota} = \bigoplus_{p+q=-1} V_{\iota}^{p,q}, \qquad \iota \in L^{\mathrm{an}},
\]
satisfying
\[
\overline{V_{\iota}^{p,q}} = V_{\bar{\iota}}^{q,p}.
\]
From now on, fix a Hermitian rational Hodge structure of weight $-1$ on $V$.

\begin{dfn}
A \emph{cycle} $S$ is an orbit of $L^{\mathrm{an}}$ under the action of $\sigma$. A cycle is called \emph{inert} if $S=\bar{S}$; otherwise it is called \emph{split}. A \emph{semi-cycle} is a subset $T \subset L^{\mathrm{an}}$ such that $T \cap \bar{T} = \varnothing$ and $T \cup \bar{T}=L^{\mathrm{an}}$.
\end{dfn}

\begin{setting}\label{hodgestructureofgaugetype}
Let $S$ be a cycle. For $\iota \in S$, assume that $V_{\iota}^{p,q}=0$ whenever $p \notin [p_{0,\iota},\, p_{1,\iota}]$. Let $\Omega \subset S$ be the subset of $\iota$ such that $p_{0,\iota}=p_{1,\iota}$, and let $\Gamma \subset S$ be its complement. Choose a bijection
\[
J_S : S \longrightarrow \{1,2,\dots,|S|\}
\]
such that $J_S(\sigma(\iota)) = J_S(\iota)+1 \bmod |S|$. Enumerate the elements of $\Gamma$ as $\gamma_1,\dots,\gamma_k$ so that
\[
J_S(\gamma_1) < J_S(\gamma_2) < \cdots < J_S(\gamma_k).
\]
For each interval $(\gamma_i,\gamma_{i+1}]$ (interpreted cyclically modulo $|S|$), define, for $w\in (\gamma_i,\gamma_{i+1}]$,
\[
t^+(w) := J_S(w) - J_S(\gamma_i), \qquad r_i^+ := J_S(\gamma_{i+1}) - J_S(\gamma_i).
\]

Define a new double grading on $\tilde{V}_S = V_S$ by setting, for $w \in S$,
\[
\tilde{V}_w^{-1,0} = \bigoplus_{m < p_{0,\gamma_i} + t^+(w)} V_{\gamma_i}^{m,n}, \qquad
\tilde{V}_w^{0,-1} = \bigoplus_{m \ge p_{0,\gamma_i} + t^+(w)} V_{\gamma_i}^{m,n},
\]
and $\tilde{V}_w^{p,q}=0$ for all other $(p,q)$.

We say that $J_S$ is \emph{admissible} with respect to $V_S=\bigoplus_{\iota\in S}V_{\iota}$ if, for every $\iota_i\in S$, one has
\[
r_i^+ \;\ge\; p_{1,\iota_i}-p_{0,\iota_i}.
\]
A collection $\{J_S\}_{S \text{ a cycle}}$ is called \emph{admissible} if each $J_S$ is admissible and, in addition,
\[
J_S = \overline{J_{\bar{S}}} \quad \text{for split cycles } S,
\]
and
\[
J_S(\iota) = -J_S(\bar{\iota}) \bmod S \quad \text{for inert cycles } S.
\]

We say that $\{J_S\}_S$ is of \emph{gauge type} if, moreover, for every semi-cycle $T \subset S$,
\[
\bigl|\{w \in [\gamma_i,\gamma_{i+1}) \mid \gamma_{i+1} \notin T\}\bigr| \equiv \sum_{w\in T} p_{0,w} \pmod{2},
\]
and, for every cycle $S$,
\[
\sum_{w\in S} p_{0,w} = 0.
\]

From now on, fix a collection of admissible bijections of gauge type.
\end{setting}

\begin{lem}[{\cite[Lemma 8.2]{bultel2008pel}}]\label{isomodv}
One can upgrade $\tilde{V}$ to a rational Hermitian Hodge structure $(\tilde{V},\tilde{\Psi},\tilde{h})$ such that, for every finite place $v$ of $L$ with $\bar{v}=v$, there exists an isomorphism
\[
(V_v,\Psi_v) \;\xrightarrow{\;\sim\;} (\tilde{V}_v,\tilde{\Psi}_v).
\]
\end{lem}

\begin{rmk}\label{moregenerality}
The setup of \cite{bultel2008pel} is in fact slightly more general than that of \ref{hodgestructureofgaugetype}. For simplicity, we work under stronger hypotheses, but all arguments extend to B\"ultel's original setting as well.
\end{rmk}

\begin{rmk}\label{givepadicsparse}
Assume $v$ is a place of $L$ with $\bar{v}=v$, and let $\sigma$ be a Frobenius lift with respect to $v$. Denote by $\mu_S$ the cocharacter obtained from the double grading on $V_S$, and let $J_S$ be an admissible bijection. Then $(V_S,\mu_S)$ is a sparse display datum.
\end{rmk}
\begin{dfn}\label{gaugetyperepresentation}
Let $G_0$ be a reductive group over $L$, and set $G = G_m \times \mathrm{Res}_{L/\mathbb{Q}} G_0$. Let 
\[
h : \mathrm{Res}_{\mathbb{C}/\mathbb{R}} G_m \longrightarrow G
\]
be a morphism. Let $\rho_i : G_0 \to \GU(V_i,\Psi_i)$ for $i=1,2,3$ be three skew-Hermitian representations of $G_0$, and denote
\[
\bigotimes_{i=1}^3 (V_i,\Psi_i,\rho_i)
\quad\text{by}\quad
(V_{123},\Psi_{123},\rho_{123}).
\]
We say that $(\rho_1,\rho_2,\rho_3)$ is of \emph{gauge type} if:

\begin{enumerate}
    \item For each $i\in\{1,2,3\}$, the skew-Hermitian Hodge structure $(V_i,\Psi_i,h_i)$ is of Hodge type, where $h_i = \rho_i \circ h$.
    
    \item There exist involution algebras $R_1,R_2,R_3,R_{123}$ that are free of finite rank over $L$, together with an action of $R_\pi$ on $(V_\pi,\Psi_\pi)$ for each 
    \[
    \pi\in\{1,2,3,\{1,2,3\}\},
    \]
    such that $G_0$ is the stabilizer of these actions inside $\prod \GU(V_\pi,\Psi_\pi)$.
    
    \item Each cycle $S$ contains an element $\tau$ such that, for each $i$, all but at most one of the Hodge numbers $\tilde{h}_{i,\tau}^{p,q}$ vanish.
    
    \item For every cycle $S$ and $1\le i\le 3$, let $(V_{i,S,v},\mu_{i,S,v})$ denote the associated sparse display datum. Then, for each $i$, there exists some $\tau \in S$ such that
    \[
    \mu_{i,S,v}^\tau = (1,1,\dots,1) 
    \quad\text{or}\quad
    (z,z,\dots,z).
    \]
\end{enumerate}
\end{dfn}

We now explain the role of these conditions:

\begin{rmk}
Condition (1) is the key hypothesis, since it ensures the existence of the Hodge-type Shimura data $(\GU(\tilde{V}_i,\Psi(V_i)),\tilde{h}_i)$. The reason for considering three representations and their tensor product, together with Conditions (2) and (3), rather than working with a single representation, is that we seek to remain within a PEL-type framework while still obtaining interesting examples that satisfy Condition (1). Condition (3) guarantees that $\GU(V_{123},\tilde{h}_{123})$ is also of Hodge type. Condition (4) is a useful technical assumption ensuring that all the displays are adjoint-nilpotent.
\end{rmk}

\begin{dfn}\label{meta-unitary Shimura data}
A \emph{meta-unitary Shimura datum} consists of
\[
(G,h,\{\rho_i,\epsilon_i,\tilde{\Psi}_i\}_{1\le i\le 3})
\]
such that:
\begin{enumerate}
    \item $\rho_1,\rho_2,\rho_3$ are skew-Hermitian representations of $G_0$ of gauge type.
    
    \item The skew-Hermitian forms $\tilde{\Psi}_i$ on $\tilde{V}_i$ are such that $(\tilde{V}_i,\tilde{\Psi}_i,\tilde{h}_i)$ are skew-Hermitian Hodge structures.
    
    \item $\epsilon_i : \mathbb{A}_{\mathbb{Q}}^\infty \otimes V_i \to \mathbb{A}_{\mathbb{Q}}^\infty \otimes \tilde{V}_i$ is an isomorphism.
\end{enumerate}

We say that a meta-unitary Shimura datum is \emph{unramified} if, in addition, there exists a self-dual $O_{L,p}$-lattice $B_i \subset V_i$ such that
\[
U(B_i,\psi_i)
\quad\text{is hyperspecial inside}\quad
U(V_i,\psi_i).
\]
\end{dfn}

\begin{example}\label{example}
Let $(G,\mu)$ be a Shimura datum that splits over $\mathbb{Q}_{p^f}$, such that $G^{\mathrm{der}}$ is simply connected, $G^{\mathrm{ad}}$ is simple and quasi-split over $\mathbb{Q}_p$, and $G$ is either of type $B_\ell$ or $C_\ell$ with $G^{\mathrm{ad}}\otimes \mathbb{R}$ having more than three times as many compact simple real factors as non-compact ones, or of type $E_7$ with $G^{\mathrm{ad}}\otimes \mathbb{R}$ having more than four times as many compact factors as non-compact ones. Then Appendix~D of \cite{bultel2008pel} constructs an explicit meta-unitary Shimura datum for $(G,\mu)$. Some of these examples are of abelian type, while others include examples of type $E_7$ and type $D$.
\end{example}

From now on, fix a level structure $K^p$ away from $p$.

\begin{dfn}\label{PEL-typemoduliproblem}
For every connected $\mathbb{F}_q$-scheme $S$ with a fixed geometric base point $s_0$,  
the set $\mathcal{S}^{\textup{tan}}_{K^p}(S)$ consists of the data:
\begin{itemize}
    \item[(1)] A polarized abelian scheme 
    \[
      (Y, \lambda, \eta, \iota)
    \]
    over $S$, equipped with
    \begin{itemize}
        \item a prime-to-$p$ level structure $\eta$ of type $K^p$,
        \item an action $\iota: \mathcal{O}_L^3 \to \End_S(Y)$,
        \item and a polarization $\lambda$
    \end{itemize}
    satisfying the usual PEL-type conditions.
    
    \item[(2)] For each $i \in \{1,2,3\}$, an action
    \[
      y_i : R_i \longrightarrow \End_S(Y_i), \qquad
      \text{where } Y_i := e_i \cdot Y
    \]
    and $e_1, e_2, e_3$ are the standard idempotents of $\mathcal{O}_L^3$.
    These actions are required to satisfy the determinant condition relative to the given $\mathcal{O}_L$-structure.
    
    \item[(3)] Let $Y_{123}$ be the tensor product of $Y_1,Y_2,Y_3$. An action of $R_{123}$ on 
    $Y_{123}$
    such that the induced action of $R_{123}$ on $Y_{123}$ is compatible with its natural action on the prime-to-$p$ adelic Tate module.  
    More precisely, the diagram
    \[
    \begin{tikzcd}
    \mathbb{A}^{\infty,p}\! \otimes V_{123} \arrow{r}{\eta} &
    \displaystyle\bigotimes_{i=1}^3 H^1_{\textup{\'et}}(Y_{i,s_0}, \mathbb{A}^{\infty,p}) \arrow{r} &
    H^1_{\textup{\'et}}(Y_{123,s_0}, \mathbb{A}^{\infty,p})
    \end{tikzcd}
    \]
    commutes with the action of $R_{123}$.
\end{itemize}
\end{dfn}
\begin{rmk}\label{represntailityofPELmodulaiproblem}
The moduli problem $\mathcal{S}^{\textup{tan}}_{K^p}$ is representable by a closed subscheme of the special fiber of the moduli space of abelian varieties with $\CO_L^3$–action, polarization, and level structure i.e the PEL-type Shimura variety $\Sh(\tilde{G},\mu^\triangleright)$ . In particular, $\mathcal{S}^{\textup{tan}}_{K^p}$ is projective.
\end{rmk}

B\"ultel constructed a candidate for the special fiber of $(G,\mu)$ as follows:

\begin{dfn}\label{specialfiber}
Let $\mathcal{S}^{\mathrm{nr}}_{F_q}$ be the scheme making the following diagram commute:
\[
\begin{tikzcd}\label{diagramreducedfiber}
\bar{\mathcal{S}}_{K^p}^{\mathrm{nr}} \arrow{r} \arrow{d} &
\bar{\CB}_{G,\mu} \arrow{d}{\Flex^{\textup{tan}}} \\
\mathcal{S}^{\textup{tan}}_{K^p} \arrow{r} & \bar{\CB}_{\tilde{G},\mu^\triangleright}^{\textup{tan}}
\end{tikzcd}
\]
We denote by $\bar{\mathcal{S}}_{K^p}$ the underlying reduced subscheme of $\mathcal{S}^{\mathrm{nr}}_{K^p}$.
\end{dfn}

\begin{lem}\label{smoothness}
The scheme $\mathcal{S}_{K^p}$ is smooth.
\end{lem}

Finally, we can define the integral model:

\begin{lem}[{\cite[Section 8]{bultel2008pel}}]\label{constructionofintegralmodel}
There exists a unique pair $(S,D)$, where $S$ is a smooth projective scheme over $W(F_q)$ lifting $\bar{S}$, and $D$ is a $(G,\mu)$-display over $S$ lifting the universal $(G,\mu)$-display over $\bar{S}$.
\end{lem}
\begin{dfn}
We define the integral model for the Shimura variety as the unique pair $(\mathcal{S}_{K^p},D)$ that lifts $\bar{\mathcal{S}}_{K^p}$ and its universal display. In other words we have a cartesian diagram:

\[
    \begin{tikzcd}
\mathcal{S}_{K^p}\arrow{r} \arrow{d} &
\CB_{G,\mu} \arrow{d} \\
\bar{\CS}_{K^p} \arrow{r} & \bar{\CB}_{G,\mu}
\end{tikzcd}
\]
\end{dfn}
Let us summarize the main properties of the integral model:

\begin{thm}[{\cite{bultel2008pel}}]\label{mainproperties}
The integral model $\CS$ has the following properties:
\begin{enumerate}
\item(Serre-Tate) The map $\CS_{K^p}\to \CB_{G,\mu}$ is formally étale. Furthermore, the deformation theory of $\CB_{G,\mu}$ is governed by Grothendieck–Messing.

    \item From the universal display over $\mathcal{S}$ we obtain a $G$-torsor $\mathcal{P}$, a map 
    $\pi : \mathcal{P} \to \Gr_{\mathcal{S},G}$, and a flat connection $\nabla$.
    
    \item The monodromy group of the connection $\nabla$ is maximal.
    
    \item  Each connected component of $\mathcal{S}$ contains an elliptic point.
    
    \item  The image of $(\mathcal{P},\pi,\nabla)$ agrees with the natural $G$-torsor, together with the filtration and connection, on $Sh(G,\mu)=\mathcal{S}_{\mathbb{Q}_p}$.
\end{enumerate}
\end{thm}
\section{Local model diagram}
In the context of meta-unitary Shimura varieties, one has a v-sheaf theoretic local model diagram, as constructed by Pappas and Rapoport in \cite[§4.9.1]{pappas2022p}. This diagram provides a global way to relate the singularities of integral model of the Shimura variety, $\mathcal{S}_K$, to the singularities of its corresponding local model.

Let $\mathcal{S}_K$ be the integral model of the Shimura variety, and let $\mathcal{S}_K^{\diamondsuit/}$ be its associated v-sheaf over $\Spd(\mathcal{O}_E)$. Let $\mathcal{P}$ be the universal $G$-shtuka over $\mathcal{S}_K$, which is bounded by the minuscule cocharacter $\mu$. The construction of the diagram begins by defining a new v-sheaf, denoted $\tilde{\mathcal{S}}_K^v$, which is a $G$-torsor over $\mathcal{S}_K^{\diamondsuit/}$.

The points of $\tilde{\mathcal{S}}_K^v$ with values in a perfectoid space $S = \Spa(R, R^+)$ are defined as pairs $(x, \alpha)$, where:
\begin{enumerate}
    \item $x: S \to \mathcal{S}_K^{\diamondsuit/}$ is a point, which determines an untilt $S^\sharp = \Spa(R^\sharp, R^{\sharp+})$ of $S$.
    \item $\alpha: G \times S^\sharp \to \mathcal{P}_\varphi(S^\sharp)$ is a $G$-isomorphism (a section or ``framing''). Here, $\mathcal{P}_\varphi(S^\sharp)$ is the pullback of the $G$-torsor $\varphi^*(\mathcal{P})$ along the canonical inclusion $S^\sharp \hookrightarrow Y_{[0,\infty)}(S) = S \,{ \dot{\times} }\, \BZ_p$.
\end{enumerate}

This construction yields a v-sheaf $\tilde{\mathcal{S}}_K^v$ equipped with two natural morphisms of v-sheaves.

First, there is a projection map which forgets the framing $\alpha$:
\[
\pi_K^v: \tilde{\mathcal{S}}_K^v \to \mathcal{S}_K^{\diamondsuit/}
\]
This morphism is a $G^\diamondsuit$-torsor, where $G^\diamondsuit$ is the v-sheaf of groups whose $S$-valued points are pairs of an untilt $S^\sharp$ and an element of $G(R^\sharp)$.

Second, a map to the v-sheaf local model $\mathcal{M}_{G,\mu}^v$ is defined. Given an $S$-valued point $(x, \alpha)$ of $\tilde{\mathcal{S}}_K^v$, the section $\alpha$ can be extended to the formal completion $\widehat{S^\sharp}$ of $S\,{\dot{\times}}\,\BZ_p$ along $S^\sharp$. Using Beauville-Laszlo glueing, the pair $(\mathcal{P}, \varphi_{\mathcal{P}} \circ \alpha)$ defines an $S$-valued point of the affine $B_{\textup{dR}}$-Grassmannian $\Gr_{G, \Spd(\mathcal{O}_E)}$. Since the shtuka $\mathcal{P}$ is bounded by $\mu$, this point lies within the v-sheaf local model $\mathcal{M}_{G,\mu}^v \subset \Gr_{G, \Spd(\mathcal{O}_E)}$. This gives a morphism
\[
q_K^v: \tilde{\mathcal{S}}_K^v \to \mathcal{M}_{G,\mu}^v
\]
which is $G^\diamondsuit$-equivariant.

Altogether, these objects form the \textbf{v-sheaf theoretic local model diagram} for $\mathcal{S}_K$:
\begin{equation}
\begin{tikzcd}
& \tilde{\mathcal{S}}_K^v \arrow[dl, "\pi_K^v"'] \arrow[dr, "q_K^v"] & \\
\mathcal{S}_K^{\diamondsuit/} & & \mathcal{M}_{G,\mu}^v
\end{tikzcd}
\label{eq:vsheaf_lmd}
\end{equation}
The existence of this diagram is a formal consequence of the existence of the G-shtuka $\mathcal{P}_K$ over $\mathcal{S}_K$, which is bounded by $\mu$.

\section{Rapoport-Papas Uniformization}

Although it is not required for the main results of this paper, we briefly outline how one can generalize the Rapoport-Papas uniformization to the setting of meta-unitary Shimura varieties. A more detailed treatment will be presented in the forthcoming work \cite{hedayatzadehrappaportlanglands}.

In \cite{hedayatzadeh2025embeddings}, we defined compatible embeddings of the regular meta-unitary Shimura variety into the Hodge-type Shimura variety $\CS_{\tilde{G},\mu^\triangleright}$ and of the associated local Shimura variety into the Rapoport-Zink space $RZ_{\tilde{G},\mu^\triangleright}$. We define the following space:
$$RZ_{G,\mu,x_0} = \hat{\CS}_K \otimes_{\CA_K} (RZ_{\tilde{G},\mu^\triangleright} \otimes \CO_E)$$

The proof of \cite[Lemma 10.4.1]{pappas2022p} applies directly in our situation. Therefore, the space $RZ_{G,\mu,x_0}$ is normal and flat, and its completion is isomorphic to $\hat{\CS}_K$. Similarly, the argument of \cite[Proposition 4.10.3]{pappas2022p} holds without major changes. The only part that requires adaptation is the proof of \cite[Proposition 4.7.1]{pappas2022p}. However, this proof relies primarily on general facts about shtukas and the Fargues-Fontaine curve, which continue to hold in our context. This allows for the construction of the required map from $\Spf R$ to $\CM_{G,\mu}^{\textup{int}}$.

The final ingredient needed for uniformization is Condition U. This is provided by \cite[Corollary 6.3]{gleason2026connected}, whose proof goes through in our situation. We note that this is possible because the analogues of both Serre-Tate and Grothendieck-Messing theory hold for meta-unitary Shimura varieties. In summary, we have the following proposition.

\begin{pro}
Let $(G,\mu)$ be a meta-unitary Shimura datum and let $k$ be a finite field. Let $x \in \CS_K(k)$ be a point, let $b_x$ be the associated $G$-isocrystal, and let $x_0$ be the corresponding point in the local Shimura variety. There exists a uniformization map
$$\Theta:\CM^{\textup{int}}_{G,b_x,\mu}\to \hat{\CS}_K$$
which restricts to an isomorphism
$$\Theta:\CM^{\textup{int}}_{G,b_x,\mu,/x_0}\to (\CS_{K,W(k),/x})^\diamondsuit$$
\end{pro}
\subsection{Local Igusa Stack}
Let $b \in B(G,\mu)$. This defines a point of the stack $\Isoc_G$ associated to $b$, which we denote by $\mathcal{D}_b$.

\begin{dfn}\label{definitionIgusastack}
 The \emph{(local) Igusa stack} associated to this stratum is the pro-étale cover of the central leaf $\mathcal{C}_b$, defined as the fibered category of pairs $(\mathcal{T}, \iota)$ consisting of a point $\mathcal{T}$ of the central leaf together with a trivialization of its associated isocrystal.

More precisely, for a perfect $\overline{\mathbb{F}}_p$-algebra $R$, we set
\[
\Igs_b(R) 
:= 
\left\{
(\mathcal{T}, \iota) \ \middle|\ 
\begin{aligned}
&\mathcal{T} \in \mathcal{C}_b(R), \\
&\iota: \CD_{b,R} \xrightarrow{\sim} \pi_{\mathrm{Crys}}(\mathcal{T}) \text{ is an isomorphism}
\end{aligned}
\right\}.
\]
The natural forgetful map $\Igs_b \longrightarrow \mathcal{C}_b$ is a pro-étale $G_b(\mathbb{Q}_p)$-torsor, where $G_b$ is the $\sigma$-centralizer of $b$.
\end{dfn}

\begin{rmk}\label{representabilityoflocalIgusahodgetype}
When $(G,\mu)$ is of compact Hodge-type, The local Igusa stack is representable by a perfect affine scheme\cite{caraiani2017generic}.
\end{rmk}
\begin{pro}\label{closedimmersionlocalIgusastack}
Let $(G,\mu)$ be a meta-unitary Shimura variety and $(\tilde{G},\mu^\triangleright)$ be the associted PEL-unitary Shimura datum. Let $[b]$ be a newton stratum of $\CS(G,\mu)$ and $[\tilde{b}]$ be the newton stratum of $\CS(\tilde{G},\mu)$ that contains the image of $[b]$. Then there is a closed immesrion $\Flex:\Igs_{b,G,\mu}\to \Igs_{b,\tilde{G},\mu^{\triangleright}}$.
\end{pro}
\begin{proof}Recall that there is a \(2\)-cartesian diagram
\[
\begin{tikzcd}
\bar{\mathcal{S}}_{K^p} \arrow{r} \arrow{d} &
\bar{\CB}_{G,\mu} \arrow{d}{\Flex^{\mathrm{tan}}} \\
\mathcal{S}^{\mathrm{tan}}_{K^p} \arrow{r} &
\bar{\CB}_{\tilde{G},\mu^\triangleright}^{\mathrm{tan}}
\end{tikzcd}
\]
where \( \Flex^{\mathrm{tan}} \) is fully faithful and representable by a closed immersion. Moreover, the morphism \( \Flex \) preserves isogeny classes of displays. It follows that we obtain a commutative diagram
\[
\begin{tikzcd}
\bar{\mathcal{S}}_{K^p} \arrow{r} \arrow{d} &
B(G,\mu) \arrow{d} \\
\mathcal{S}^{\mathrm{tan}}_{K^p} \arrow{r} &
B(\tilde{G},\mu^\triangleright)
\end{tikzcd}
\]
which induces a closed immersion
\[
\Flex \colon C_{G,b} \hookrightarrow C_{\tilde{G},b}
\]
between the central leaves.

On the other hand, the diagram
\[
\begin{tikzcd}
\Igs_{G,b} \arrow{r} \arrow{d} &
C_{G,b} \arrow{d} \\
\Igs_{\tilde{G},b} \arrow{r} &
C_{\tilde{G},b}
\end{tikzcd}
\]
is cartesian. Indeed, using that \( \Flex \) is fully faithful, giving a point \( \mathcal{T} \) in \( \mathcal{C}_b \) together with an isomorphism
\[
\iota \colon \mathcal{D}_{b,R} \xrightarrow{\sim} \pi_{\mathrm{Crys}}(\mathcal{T})
\]
is equivalent to giving the pair \( (\mathcal{T}, \Flex(\iota)) \).
\end{proof}
\begin{cor}\label{represntabilityIgusastack}
Let $(G,\mu)$ be a meta-unirary Shimura variety and $[b]$ be a newton stratum, then $\Igs_b$ is representable by a perfect affine scheme.
\end{cor}
\begin{proof}
This follows from \ref{representabilityoflocalIgusahodgetype} and \ref{closedimmersionlocalIgusastack}.
\end{proof}
\section{Igusa stack}

Let $(G, \mu)$ be a Shimura datum with reflex field $E$, and let $K_p \subset G(\mathbb{Q}_p)$ be a hyperspecial compact open subgroup. We consider levels of the form $K = K_p K^p$, where $K^p$ varies over compact open subgroups of $G(\mathbb{A}_f^p)$.

Let $\mathcal{P}^{\mathrm{rig}}$ denote the universal $G$-shtuka defined over the generic fiber $\Sh_K(G, \mu)$. Consider a system $(\mathcal{S}_{K^p}, \mathcal{P})_{K^p}$, where:
\begin{itemize}
  \item for each $K^p \subset G(\mathbb{A}_f^p)$, $\mathcal{S}_{K^p}$ is an integral model of $\Sh_K(G, \mu)$ equipped with finite étale transition maps;
  \item Let $\mathcal{P}$ be a $G$-shtuka over $\mathcal{S}_K$ extending $\mathcal{P}^{\mathrm{rig}}$ (corresponding to a morphism $\pi_{\mathrm{Crys}} \colon \mathcal{S}_K \to \Sht_{G,\mu}$). Furthermore, assume that $\mathcal{P}$ comes from a $(G,\mu)$-display $\mathcal{D}$. (If \cite[Conjecture 8.8]{hedayatzadeh2024deformations} is true, then $\mathcal{P}$ comes from a $(G,\mu)$-display on $\mathcal{S}_K$.) 
\end{itemize}
The tower of level structures induces pro-étale $\ell$-adic local systems $\mathcal{L}_\ell$ on $\mathcal{S}_K$. Furthermore, we assume the following extension property: for every perfectoid pair $(R,R^+)$, every point $x \in \mathcal{S}_K^\diamondsuit(R,R^+)$, and every quasi-isogeny $\phi \colon \mathcal{P}_x \to \mathcal{P}'$, there exists a unique point $x' \in \mathcal{S}_K^\diamondsuit(R,R^+)$ such that for each $\ell \neq p$, the local systems $\mathcal{L}_{\ell,x}$ and $\mathcal{L}_{\ell,x'}$ are isomorphic.

Suppose that $(\mathcal{S}_K, \mathcal{P})$ satisfies the \emph{Serre--Tate lifting property}: we assume that the map $\mathcal{P}:\CS\to \CB^{\perf}(G,\,\mu)$ is formally \'etale.
\begin{pro}[Rigidity of Quasi-Isogenies]\label{quasi-isogeny over nilpotent thickening}
In the setting above, suppose that $S' \to S$ is a surjective morphism of rings such that $p$ is nilpotent in $S'$ and the kernel of the morphism is nilpotent. Then the restriction base change from $S'$ to $S$ induces an equivalence between the groupoids of points of $\mathcal{S}$ valued in $S'$ and in $S$, considered up to quasi-isogeny.
\end{pro}

\begin{proof}
It suffices to establish the corresponding statement for the stack $\mathcal{B}^{\perf}_{G,\mu}$. By the deformation theory of $(G,\mu)$-displays, every object over $S$ admits a lift to $S'$. The obstructions to lifting and the deformations are controlled by the Hodge filtration on the associated crystal. While the lift as an integral object depends on the choice of filtration, the underlying iso-crystal with $G$-structure is rigid. Consequently, all such lifts are canonically quasi-isogenous to one another.
\end{proof}

\begin{example}\label{example:integral_model_prop}
The canonical integral models of Shimura varieties of abelian type and meta-unitary Shimura varieties satisfy this condition. Let us explain this in more detail in the case of meta-unitary Shimura varieties. Consider the commutative diagram

\[
\begin{tikzcd}
\bar{\mathcal{S}}_{K^p} \arrow{r} \arrow{d} &
\bar{\CB}_{G,\mu} \arrow{d}{\Flex^{\textup{tan}}} \\
\mathcal{S}^{\textup{tan}}_{K^p} \arrow{r} & \bar{\CB}_{\tilde{G},\mu^\triangleright}^{\textup{tan}}
\end{tikzcd}
\]

A quasi-isogeny between the points $x$ and $x'$ consists of a quasi-isogeny between their associated displays $\CP_x$ and $\CP_{x'}$, together with a quasi-isogeny between their associated abelian varieties $\CA_x$ and $\CA_{x'}$, such that both induce the same image in the category of quasi-isogenies of $\CB^\tan$. Such a quasi-isogeny is said to be \emph{at $p$} if the degree of the quasi-isogeny between the abelian varieties is a power of $p$.

More precisely, let $x=(A,\CD)$ and $x'=(A',\CD')$ be two points of $\bar{S}$. By a quasi-isogeny between $x$ and $x'$ we mean a quasi-isogeny $\tau$ between $A$ and $A'$ together with a quasi-isogeny $\rho$ between $\CD$ and $\CD'$ such that $\Flex(\rho)$ and $\tau$ are compatible with each other.
\end{example}

\begin{dfn}\label{dfn:formal_quasi_isogeny}
Let $(R,R^+)$ be a perfectoid ring with pseudo-uniformizer $\varpi$. A \emph{formal quasi-isogeny} between two points $x,y \in \mathcal{S}(R,R^+)$ is a quasi-isogeny between their reductions $x_{R^+/\varpi}$ and $y_{R^+/\varpi}$ over $R^+/\varpi$.
\end{dfn}
In this setting, we proceed to define the Igusa stack.

\begin{dfn}\label{dfn:prestack_Igusa}
Let $\mathbb{F}_q$ be the residue field of the reflex field $E$. Consider the prestack $\Igs_{G,\mu,K^p}^{\mathrm{pre}}$ over $\Spd(\mathbb{F}_q)$ that assigns to a perfectoid affinoid pair $(R, R^+)$ over $\Spd(\mathbb{F}_q)$, with pseudo-uniformizer $\varpi$, the groupoid whose objects are points $x \in \mathcal{S}(R, R^+)$, and whose morphisms are formal quasi-isogenies. Let $\Igs_{G,\mu,K^p}$ denote the stackification of $\Igs_{G,\mu,K^p}^{\mathrm{pre}}$ in the $v$-topology.
\end{dfn}

\begin{rmk}\label{frobeniusisidentity}
There is a natural transformation between the absolute Frobenius and the identity on the prestack $\Igs_{G,\mu,K^p}^{\mathrm{pre}}$. This allows us to descend the stack $\Igs_{G,\mu,K^p}$ to $\Spd(\mathbb{F}_p)$.
\end{rmk}

\begin{rmk}\label{rmk:indep_uniformizer}
By the rigidity of quasi-isogenies (Proposition \ref{quasi-isogeny over nilpotent thickening}), the category $\Igs_{G,\mu,K^p}^{\mathrm{pre}}(R, R^+)$ is independent of the choice of the pseudo-uniformizer $\varpi$.
\end{rmk}

\begin{dfn}\label{dfn:Igusa_points}
Let $(R,R^+)$ be a perfectoid pair over $\BZ_p$ with pseudo-uniformizer $\varpi$. We define the groupoid of points $\Igs_{G,\mu,K^p}(R,R^+)$ to be $\Igs_{G,\mu,K^p}(R^+/\varpi)$. In this way, we consider the Igusa stack $\Igs_{G,\mu,K^p}$ as a $v$-stack over $\Spd(\BZ_p)$.
\end{dfn}

\begin{rmk}\label{rmk:reduction_map}
If the Shimura variety $\Sh_{K}(G,\mu)$ is proper, it identifies with the adic generic fiber of the formal scheme $\CS_K$. Consequently, we obtain a reduction morphism $\red \colon \CS_K \to \Igs_{G,\mu,K^p}$ which sends a point of $\CS_K$ to its reduction modulo $p$.
\end{rmk}

\begin{construction}\label{const:pi_crys}
There exists a map $\bar{\pi}_{\mathrm{Crys}}^\diamondsuit \colon \Igs_{G,\mu,K^p} \to \Bun_G$. It suffices to define this map for the prestack $\Igs_{G,\mu}^{\mathrm{pre}}$. Let $x \in \Igs_{G,\mu}^{\mathrm{pre}}(R,R^+)$. Let $y$ be a point of $\CS(R,R^+)$ that lifts $x$ and let $\mathcal{P}_y$ be the associated shtuka. Consider the underlying shtuka with no leg $\CP_{y,\infty}$, which defines a $G$-bundle on the Fargues--Fontaine curve. A formal quasi-isogeny of $G$-Shtukas $\mathcal{P}_{y} \to \mathcal{P}_{y'}$ induces an isomorphism between the $\CP_{y,\infty}$ and $\CP_{y',\infty}$. Therefore, we obtain a well-defined point of $\Bun_G$.
\end{construction}
We begin by recalling the diagram from~\ref{diagramreducedfiber}. The functor $\Flex$ respects quasi-isogenies, and therefore induces a well-defined map on the corresponding stacks:
\[
\Igs_{G,\mu} \to \Igs_{\tilde{G},\tilde{\mu}}.
\]

\begin{lem}\label{fromshtukatodisplay1}
Let $\mathcal{D}_0$ be a $(G,\mu)$-display over an integral perfectoid ring $R^+$, and let $\mathcal{P}_0 = \Fib(\mathcal{D}_0)$ be its associated shtuka. Suppose $\mathcal{P}$ is another shtuka bounded by $\mu$ such that there exists an isomorphism
\[
\iota: \mathcal{P}_{0,[r,\infty]} \cong \mathcal{P}_{[r,\infty]}
\]
for some real number $r>0$. Then $\mathcal{P}$ is the shtuka associated to a $(G,\mu)$-display $\mathcal{D}$.
\end{lem}

\begin{proof}
The construction of the display $\mathcal{D}$ is equivalent to producing a $G$-torsor over $A_{\inf}(R^+)$. We construct this torsor by gluing. Specifically, we take the $G$-torsor associated to $\mathcal{D}_0$ over the space $\Spa A_{\inf}(R^+)/Y_{[r,\infty]}$ and glue it to the $G$-torsor associated to $\mathcal{P}$ over $Y_{[r,0)}$ using the isomorphism $\iota$ as the transition map.
\end{proof}

\begin{lem}\label{fromshtukatodisplay}
Let $i \colon G \to \tilde{G}$ be a closed immersion of reductive groups. 
Let $\mathcal{D}$ be a point of $\Sht_{G,\mu}$ over a product of points, 
such that $\pi_{\mathrm{Crys}(\mathcal{D})}$ lifts to a point of $\CB_{\tilde{G},\mu}^{\perf}$. 
Then $\mathcal{D}$ arises from a point of $\CB_{\tilde{G},\mu}^{\perf}$.
\end{lem}

\begin{proof}
It suffices to lift the underlying $G$-torsor of the display $\mathcal{D}$ over $\mathcal{Y}$ 
to a $G$-torsor over $\mathbb{A}_{\inf}$. 

To give a $G$-torsor over $\mathbb{A}_{\inf}$, it is enough to give a $\tilde{G}$-torsor 
over $\mathbb{A}_{\inf}$ together with suitable tensors cutting out a reduction to $G$. 

In our situation, we are given a $\tilde{G}$-torsor over $\mathbb{A}_{\inf}$ and 
compatible tensors over $\mathcal{Y}$. By the fully faithful functor from 
$G$-torsors over $\mathbb{A}_{\inf}$ to $G$-torsors over $\mathcal{Y}$, 
these tensors extend uniquely to $\mathbb{A}_{\inf}$. 

This produces the desired $G$-torsor over $\mathbb{A}_{\inf}$, hence the claim.
\end{proof}

From now on, we concentrate on the meta-unitary Shimura variety. Let $(G,\mu)$ be a meta-unitary Shimura datum, and let $(\tilde{G},\mu^{\triangleright})$ be the auxiliary PEL-type unitary Shimura datum introduced in Section~\ref{meta-unitaryShimurasection}. Recall that there is a closed immersion $G \hookrightarrow \tilde{G}$.

\begin{lem}\label{rmk:v_cover}
The natural map $\bar{S}_{G,\mu^\triangleright,K}\to \Igs_{G,\mu,K^p}$ is a $v$ cover, and for every point of $x$ of the Igusa stack over a product of points, the fiber over $x$ consists of displays $\CP_y$ such that $\CP_{y,\infty}= \pi_{\mathrm{Crys}}(x)$.
\end{lem}
\begin{proof}
We know that this is true for the map $\bar{S}_{\tilde{G},\mu^\triangleright,K}\to \Igs_{\tilde{G},\mu^\triangleright,K^p}$ by \cite{daniels2024Igusa}. 

Consider a point $\bar{x}$ of the stack $\Igs_{G,\mu}$ over the product of points. One can extend $\Flex \bar{x}$ to a point $\Flex x$ of $\bar{S}_{\tilde{G},\mu^\triangleright,K}$. 
By proposition \ref{existanceoflattice}, one can lift $\bar{\pi}_{\textup{Crys}}(\bar{x})$ to a point $\CD$ of $\Sht_{G,\mu}$. The pair $\Flex x$ and $\CD$ will give us a point of $\bar{S}_{G,\mu,K}$.
\end{proof}
\begin{lem}\label{functorial}
The natural map $\Igs_{G,\mu}\to \Igs_{\tilde{G},\mu^\triangleright}$ is a closed immersion.
\end{lem}
\begin{proof}
Consider the cartesian diagram
\begin{center}
\begin{tikzcd}
\bar{S}_{G,\mu,K} \arrow[r]\arrow[d] &  \bar{S}_{\tilde{G},\mu^\triangleright,K} \arrow[d] \\
\Igs_{G,\mu,K^p} \arrow[r] & \Igs_{\tilde{G},\mu,K^p}
\end{tikzcd}
\end{center}
the vertical maps are $v$-covers and the map $\bar{S}_{G,\mu,K}\to \CS_{\tilde{G},\mu^\triangleright,K}$ is a closed immersion, therefore the map $\Igs_{G,\mu,K^p}\to \Igs_{\tilde{G},\mu,K^p}$ is also a closed immersion.
\end{proof}

\begin{thm}\label{pro:cartesian_diagram characteristic p}
If $(G,\mu)$ is a meta-unitary Shimura variety, then the diagram below is cartesian:
\begin{center}
\begin{tikzcd}
\bar{\mathcal{S}}_K^\diamondsuit \arrow{r} \arrow{d} & \CB^{\perf}_{G,\mu} \arrow{d} \\
\Igs_{G,\mu,K^p} \arrow{r} & \Bun_{G,\mu^{-1}}
\end{tikzcd}
\end{center}
\end{thm}

\begin{proof}
Let $F$ denote the fiber product of the diagram under consideration. We have a natural map
\[
\bar{\mathcal{S}}_{G,\mu} \to F,
\]
and we aim to show that it is an isomorphism. Since products of geometric points form a basis for the $v$-topology, it suffices to prove this claim after base changing to such a product.

We first establish surjectivity. Let a point of $F$ over a geometric point $(C,C^+)$ be given. Such a point consists of a tuple
\[
(A_0, \mathcal{D}_0, \mathcal{P}, \alpha, \rho),
\]
where $\rho$ is an isomorphism $\rho: \Fib(\mathcal{D}_0) \otimes_{C^+/w} O_C/w \to \mathcal{P} \otimes_{C^+/w} O_C/w$. From this data, we must construct a corresponding point of the meta-unitary Shimura variety $\bar{\mathcal{S}}_{G,\mu}$. This requires constructing a compatible pair consisting of a formal abelian scheme $\mathcal{A}$ with a $\tilde{G}$-structure and a trivialization of its Tate module, along with a $(G,\mu)$-display $\mathcal{D}$.

The existence of such a lift is guaranteed in principle by the fiber product conjecture for Hodge-type Shimura varieties, which ensures that the image of our data under the functor $\Flex$ lifts to a point of $\bar{\mathcal{S}}_{\tilde{G},\mu^\triangleright}$. We now make this construction explicit.

First, applying Lemma~\ref{fromshtukatodisplay1}, we obtain a $(G,\mu)$-display $\mathcal{D}$ from the shtuka $\mathcal{P}$. Second, the work of Zhang~\cite{zhang2023pel} provides a formal abelian scheme $\mathcal{A}$ (with its requisite $\tilde{G}$-structure) whose associated $p$-divisible group is isomorphic to $\Flex(\mathcal{D})$. The resulting pair $(\mathcal{A}, \mathcal{D})$ provides the required lift to $\bar{\mathcal{S}}_{G,\mu}$, confirming the surjectivity of the map on geometric points.

For full faithfulness, we again reduce to the corresponding statement for $\bar{\mathcal{S}}_{\tilde{G},\mu^\triangleright}$, where it is known. This follows from the fully faithful maps
\[
\CB^{\perf}_{G,\mu} \to \CB^{\perf}_{\tilde{G},\mu^\triangleright}
\quad \text{and} \quad
\Igs_{G,\mu} \to \Igs_{\tilde{G},\mu^\triangleright}.
\]
\end{proof}
\begin{thm}\label{thm:cartesian_diagram}
If $(G,\mu)$ be a meta-unitary Shimura variety, then the diagram below is two cartesian:
\begin{center}
\begin{tikzcd}
\CS_{K}^\diamondsuit \arrow[r, "\pi_{\mathrm{Crys}}"] \arrow[d, "\red"'] &  \Sht_{G,\mu} \arrow[d, "\Fib"] \\
\Igs_{G,\mu,K^p} \arrow[r, "\bar{\pi}_{\mathrm{Crys}}"] & \Bun_{G,\mu^{-1}}
\end{tikzcd}
\end{center}
\end{thm}
\begin{proof}
Consider the cartesian diagram
\begin{center}
\begin{tikzcd}
\mathcal{S}_K^\diamondsuit \arrow{r} \arrow{d} &
\CB_{G,\mu} \arrow{d} \\
\bar{\mathcal{S}}_K^\diamondsuit \arrow{r} & \bar{\CB}_{G,\mu}
\end{tikzcd}
\end{center}
and combine it with the Cartesian diagram:
\begin{center}
\begin{tikzcd}
\bar{\mathcal{S}}_K^\diamondsuit \arrow{r} \arrow{d} &
\CB_{G,\mu} \arrow{d} \\
\Igs_{G,\mu,K^p} \arrow{r} & \Bun_{G,\mu^{-1}}
\end{tikzcd}
\end{center}
\end{proof}
\begin{rmk}\label{rmk:generic_fiber_cartesian}
Over the generic fiber, we have the cartesian diagram
\begin{center}
\begin{tikzcd}
\Sht_{G,\mu} \arrow[r, "\pi_{HT}"] \arrow[d, "\red"'] &  {[\Gr_{G,\mu^{-1}}/K^p]} \arrow[d, "\text{BL}"] \\
\bar{\Sht}_{G,\mu} \arrow[r, "\bar{\pi}_{\mathrm{Crys}}"] & \Bun_{G,\mu^{-1}}
\end{tikzcd}
\end{center}
Therefore we have the cartesian diagram at the hyperspecial level:
\begin{center}
\begin{tikzcd}
\Sh(G,\mu)^\diamondsuit_K \arrow[r, "\pi_{HT}"] \arrow[d, "\red"'] &  {[\Gr_{G,\mu^{-1}}/K_p]} \arrow[d, "\text{BL}"] \\
\Igs_{G,\mu,K^p} \arrow[r, "\bar{\pi}_{\mathrm{Crys}}"] & \Bun_{G,\mu^{-1}}
\end{tikzcd}
\end{center}
Now let $K'_p\subset K_p$ be an arbitrary level at $p$ and set $K'=K'_pK^p$, we have the cartesian diagram:
\begin{center}
    \begin{tikzcd}
\Sh(G,\mu)^\diamondsuit_{K'} \arrow[r, "\pi_{\mathrm{Crys}}"] \arrow[d, "\red"'] &  \Sht_{G,\mu,K'_p} \arrow[d,] \\
\Sh(G,\mu)^\diamondsuit_K \arrow[r, "\bar{\pi}_{\mathrm{Crys}}"] & \Sht_{K_p}
\end{tikzcd}
\end{center}
Therefore we get the cartesian diagram:
\begin{center}
\begin{tikzcd}
\Sh(G,\mu)^\diamondsuit_{K'} \arrow[r, "\pi_{HT}"] \arrow[d, "\red"'] &  {[\Gr_{G,\mu^{-1}}/K'_p]} \arrow[d, "\text{BL}"] \\
\Igs_{G,\mu,K^{p}} \arrow[r, "\bar{\pi}_{\mathrm{Crys}}"] & \Bun_{G,\mu^{-1}}
\end{tikzcd}
\end{center}
\end{rmk}
At every parahoric level $\CG$, we can also define $\CS_{\CG K^p}$ as the cartesian product:

\begin{center}
\begin{tikzcd}
\CS_{\CG K^p} \arrow[r, "\pi_{HT}"] \arrow[d, "\red"'] &  \Sht_{\CG} \arrow[d, ] \\
\Igs_{G,\mu,K^{p}} \arrow[r, "\bar{\pi}_{\mathrm{Crys}}"] & \Bun_{G,\mu^{-1}}
\end{tikzcd}
\end{center}

The definition of $\Igs$ involves a stackification, but proposition \ref{existanceoflattice} gives us a direct interpertation of the points over the product of fields:

\begin{lem}\label{descriptionofpoints}
If $S$ is a product of points then the map $\CS_{G,\mu,K^p}(S)$ to $\Igs_{K^p}(S)$ is essentially surjective.
\end{lem}

\section{Some Cohomological Consequences}
The validation of the Fiber Product Conjecture provides the necessary framework to deduce cohomological results, a strategy previously employed by \cite{zhang2023pel} and \cite{daniels2024Igusa}. Consequently, we obtain these cohomological consequences for meta-unitary Shimura varieties.

\begin{cor}\label{cor:igs_properties}
Let $(G,\mu)$ be a meta-unitary Shimura datum.
\begin{enumerate}
    \item The v-stack $\Igs_{G,\mu,K^p}$ is a $\ell$-cohomologically smooth small Artin v-stack of dimesnion $0$ and with dualizing sheaf $\Lambda[0]$.
    \item The morphism $\bar{\pi}_{\mathrm{Crys}}:\Igs_{G,\mu,K^p}\to \Bun_G$ is seperated, represenatble in spatial diamonds, compactifiable of finite transcendental dimension.
    \item for any ring of coefficients $\Lambda$ with $n\Lambda=0$, $\gcd(n,p)=1$, we have a natural equivalence $$BL^*\bar{\pi}_{Crys,*}\longrightarrow \pi_{HT,*}\red^*$$ as functors from $\CD(\Igs_{G,\mu,K^p},\Lambda)$ to $\CD(\Gr_{G,\mu^{-1},K^p},\Lambda)$.
    \item The complex $\bar{\pi}_{\mathrm{Crys},*}\Lambda$ lies in $\CD^{\ULA}(\Bun_G)$
\end{enumerate}
\end{cor}
\begin{proof}
The proofs are analogous to those presented in \cite{zhang2023pel} and \cite{daniels2024Igusa}. The arguments in \cite{zhang2023pel} rely primarily on fiber product formulas and general properties of shtukas and the stack $Bun_G$. In our setting, these arguments are simplified by the fact that the relevant Shimura varieties and Igusa stacks are compact.
\end{proof}

Consider the Newton stratum $\Bun_G^b$, and let $\Igs^b$ and $\Sht^b$ be the pullbacks to this stratum. Consider the local Shimura variety $\CM^{\textup{int}}_{G,\mu,b}$.

\begin{lem}\label{lem:torsors}
The natural maps $\Ig_b\to \Igs^b$ and $\CM^{\textup{int}}_{G,\mu,b} \to \Sht^b$ are $\tilde{G}_b$-torsors, and we have the identifications $\Igs^b=[\Ig_b/\tilde{G}_b]$ and $\Sht^b=[\CM^{\textup{int}}_{G,\mu,b}/\tilde{G}_b]$.
\end{lem}
\begin{proof}
We briefly sketch how the arguments of \cite{daniels2024Igusa} can be adapted to our setting. First, the 2-Cartesian diagram in \cite[§4.3.2]{daniels2024Igusa} is a consequence of fiber product formulas. Moreover, \cite[Lemma 4.4.2]{daniels2024Igusa} is a general result that holds without any specific assumptions on the group $G$. The subsequent results, namely \cite[Lemmas 4.4.3, 4.4.5, 4.5.1, 4.5.3]{daniels2024Igusa}, are formal consequences of these facts.

The main ingredient for the proof of \cite[Lemma 8.5.2]{daniels2024Igusa} is \cite[Lemma 11.25]{zhang2023pel}, which relies on the fact that the group $\tilde{G}_b$ is isomorphic to the automorphism group of the associated $p$-divisible group. This property holds in our context. Finally, \cite[Lemma 8.5.3]{daniels2024Igusa} follows from \cite[Lemma 4.4.5]{daniels2024Igusa} combined with standard properties of the cohomology of diamonds and the stack $Bun_G$.
\end{proof}

Using this lemma, we obtain the geometric incarnation of the Mantovan formula for the meta-unitary Shimura varieties.

\begin{cor}\label{mantovanformulegeometric}
The integral Shimura variety for a meta-unitary Shimura datum admits a Newton stratification by $\CS_K^b$, and we have
\[ \CS_K^b = [\Ig_b/\tilde{G}_b] \times_{[*/\tilde{G}_b]} [\CM^{\textup{int}}_{G,\mu,b}/\tilde{G}_b]. \]
\end{cor}
\begin{proof}
The proof of \cite[Theorem 8.5.7]{daniels2024Igusa} remains valid in our situation. It is a formal consequence of Lemma~\ref{lem:torsors} and general properties of the dualizing sheaf of $Bun_G$, which hold without any specific assumptions on the group $G$.
\end{proof}

Recall that by \cite{hamann2025dualizing} the dualizing complex on $Bun_{G,k}^b$ is isomorphic to $\delta_b^{-1}[-2d_b]$ where $\delta_b$ is a cocharcter and $d_p=<2\rho,v_b>$. The proof of the following theorems is exactly the same as \cite[Theorem 8.4.10 and Theorem 8.5.7]{daniels2024Igusa} because those proofs does not use that $(G,\mu)$ is of Hodge type:

\begin{thm}\label{heckeoperatorshimuracohomology}
There is a Hecke and Galios equivariant isomorphism $$\textup{R}\Gamma(\Sh_{K_p, E}, \Lambda)\simeq R h_{2,k,*} ( h_{1,k}^* A [d](d/2) ) $$
\end{thm}

\begin{thm}\label{thm:mantovan_cohomological}
Let $(G,\mu)$ be a meta-unitary Shimura datum. There exists a filtration on the complex of smooth  $G(\mathbb{Q}_p)\times W_E$ representations $\textup{R}\Gamma(\Sh_{K_p}, \Lambda)$, indexed by the set of Newton strata $[b] \in B(G)$, whose graded pieces are given by
\[ \textup{gr}_{[b]} \textup{R}\Gamma(\Sh_{K_p, E}, \Lambda) \simeq \textup{R}\Gamma(\Igs^{b}_{G,\mu,K^p}, \Lambda)^{\textup{op}} \otimes_{\CH(G_b)}^L \textup{R}\Gamma_c(\CM_{G,b,\mu,\infty}, \delta_b)[2d_b]. \]
\end{thm}
\section{compatability for Fargues-Scholze local Langlands}
One can reinterpret Theorem \ref{heckeoperatorshimuracohomology} as establishing compatibility between the Fargues-Scholze Local Langlands correspondence and the cohomology of Shimura varieties. We will follow \cite[Section 9]{daniels2024Igusa}.

Let $r \colon \hat{G} \to \operatorname{GL}(V)$ be a representation of the dual group, and consider the induced vector bundle $\mathcal{V}$ on the stack of $L$-parameters. The spectral action of $\mathcal{V}$ on $D(\operatorname{Bun}_G)$ coincides with that of the Hecke operator $T_V$. 

Let $\mathcal{Z}^{\mathrm{spec}}(G,\Lambda)$ be the ring of global sections on the stack of $L$-parameters, and let $\mathcal{Z}(G(\mathbb{Q}_p),\Lambda)$ be the Bernstein center of the category of representations. There exists a morphism 
\[
    \Psi_G \colon \mathcal{Z}^{\mathrm{spec}}(G,\Lambda) \to \mathcal{Z}(G(\mathbb{Q}_p),\Lambda).
\]
An irreducible representation of $G(\mathbb{Q}_p)$ yields a character $\chi$ of the Bernstein center. Under the Fargues--Scholze Local Langlands correspondence, we attach the local system associated to $\chi \circ \Psi_G$ to this character. The proof of the following theorem only uses inputs established in the last section and some general facts about the geometrization of local Langlands that are true for any $G$.

\begin{thm}\cite[Theorem 9.1.4]{daniels2024Igusa}
Let $(G,\mu)$ be a meta-unitary Shimura datum. The spectral and the usual Weil action on $\textup{R}\Gamma(\Sh_{K_p}, \Lambda)$ agree. Moreover, The action of $\mathcal{Z}^{\mathrm{spec}}(G,\Lambda)$ on $\textup{R}\Gamma(\Sh_{K_p}, \Lambda)$ factors through  the natural action of $\mathcal{Z}(G(\mathbb{Q}_p),\Lambda)$ via $\Psi_G$.
\end{thm}

Denote by $\mathcal{Z}_{K_p}$ the center of $\mathcal{H}_{K_p}$. For a character $\chi : \mathcal{Z}_{K_p} \to L$, define the associated $W_E$-representation
\[
W^i(\chi) := H^i_{\textup{\'et}}(\Sh_K(G,\mu), \Lambda) \otimes_{\mathcal{Z}_{K_p}, \chi} L.
\]
Using this, one can obtain the following compatibility theorem:

\begin{thm}\label{compatabilitywithScholzefargue}\cite{daniels2024Igusa}
Let $(G,\mu)$ be a meta-unitary Shimura datum, and assume that the order of $\pi_0(Z(G))$ is prime to $\ell$. Then every irreducible $L$-linear representation of $W_E$ appearing as a subquotient of $W^i(\chi)$ also appears as a subquotient of $(r_\mu \circ \varphi_\chi)|_{W_E}$.
\end{thm} 

\section{Functoriality of Igusa stack}
In this Section we will prove that the Igusa stack constructed for meta-unitray Shimura varieties is functorial in the sense of \cite{kim2025uniqueness}.

\begin{dfn}\label{dfn:functorial_Igusa}
Let \( U \subseteq \Sh(G,X)_E^\diamondsuit \) be an open subsheaf stable under the action of \( G(\mathbb{A}_f^{p,\infty}) \). A functorial Igusa stack for \( U \) is the data of
\begin{itemize}
    \item a v-stack $\Igs^U(G, X)$ together with a \( G(\mathbb{A}_f^{p,\infty}) \)-action,
    \item a \( G(\mathbb{A}_f^{p,\infty}) \)-equivariant map of v-stacks \( \pi_{\mathrm{HT}}^U \colon \Igs^U(G, X) \to \Bun_{G,\mu^{-1}} \), where the action on \( \Bun_{G,\mu^{-1}} \) is trivial,
    \item a 2-Cartesian diagram of v-stacks
    \[
    \begin{tikzcd}
        U \arrow[r, "\pi_{\mathrm{HT}}"] \arrow[d, "i_{\Igs}"'] & \Gr_{G,\mu^{-1},E} \arrow[d, "\textup{BL}"] \\
        \Igs^U(G, X) \arrow[r, "\pi_{\mathrm{HT}}^U"] & \Bun_{G,\mu^{-1}},
    \end{tikzcd}
    \]
\end{itemize}
such that
\begin{enumerate}
    \item the \( G(\mathbb{A}_f^{p,\infty}) \)-action on \( U \simeq \Igs^U(G, X) \times_{\Bun_{G,\mu^{-1}}} \Gr_{G,\mu^{-1},E} \) induced by the \( G(\mathbb{A}_f^{p,\infty}) \)-action on \( \Igs^U(G, X) \) and the \( G(\mathbb{Q}_p) \)-action on \( \Gr_{G,\mu^{-1},E} \) recovers the Hecke action on \( U \),
    \item the \( (\phi, \mathrm{id}) \)-action on \( \Igs^U(G, X) \times_{\Bun_{G,\mu^{-1}}} \Gr_{G,\mu^{-1},E} \), induced by the canonical isomorphism \( \phi_{\Bun_{G,\mu^{-1}}} \simeq \mathrm{id}_{\Bun_{G,\mu^{-1}}} \), recovers the identity map on \( U \).
\end{enumerate}
\end{dfn}
\begin{thm}\label{functorialIgusa}
The Igusa stack for the meta-unitary Shimura variety is functorial in the sense of the definition \ref{dfn:functorial_Igusa}.
\end{thm}
\begin{proof}
The above axioms are true by theorem \ref{thm:cartesian_diagram} and remark \ref{frobeniusisidentity}.
\end{proof}

\begin{cor}\label{existanceofIgusaforembedding}
Let $(G,\mu)$ be a Shimura datum such that there exists an embedding of $(G,\mu)$ inside a meta-unitary Shimura datum. Then the functorial Igusa stack for $(G,\mu)$ exists.
\end{cor}
\begin{proof}
    Combine Theorems \ref{functorialIgusa} and \cite[Theorem 11.4]{kim2025uniqueness}.
\end{proof}
\section{Generic part of the cohomology of meta-unitary Shimura variety}

To obtain a precise understanding of the cohomology of meta-unitary Shimura varieties at hyperspecial level, we adopt the methods established by Yang and Zhu\cite{yang2025generic}. Their arguments regarding the geometric and spectral description of cohomology apply verbatim in our setting. Furthermore, the situation is considerably simplified by the fact that meta-unitary Shimura varieties are proper. Consequently, the technical complications regarding the boundary of the Shimura variety—specifically, the need to verify that the Igusa stack is ``well-positioned'' relative to compactifications—vanish in our context.

\subsection{Unipotent Categorical Local Langlands Correspondence}
The central engine of the Yang–Zhu method is the unipotent categorical local Langlands correspondence. In this section, we will give a brief summary of the key objects and the main theorem of this correspondence. For further details, we refer the reader to~\cite{yang2025generic} and~\cite{hemo2021unipotent}.

Let $\Lambda$ be the coefficient ring ($\overline{\mathbb{Q}}_\ell$ or $\overline{\mathbb{F}}_\ell$). Let $\Shv^{\mathrm{unip}}_{\mathrm{f.g.}}(\Isoc_G, \Lambda)$ denote the category of ind-finitely generated sheaves on $\Isoc_G$ that are unipotent on each Newton stratum. One can define an exotic $t$-structure on this category:
\begin{dfn}\label{dfn:exotic_t_structure}
    The \emph{exotic $t$-structure} on $\Shv(\Isoc_G, \Lambda)$, denoted by $$(\Shv(\Isoc_G, \Lambda)^{e, \leq 0}, \Shv(\Isoc_G, \Lambda)^{e, \geq 0})$$, is defined by the following conditions on an object $\mathcal{F}$:
    \begin{enumerate}
        \item $\mathcal{F} \in \Shv(\Isoc_G, \Lambda)^{e, \leq 0}$ if and only if for all $b \in B(G)$,
        \[ (i_b)^! \mathcal{F} \in \Rep(G_b(E), \Lambda)^{\leq \langle 2\rho, \nu_b \rangle}. \]
        \item $\mathcal{F} \in \Shv(\Isoc_G, \Lambda)^{e, \geq 0}$ if and only if for all $b \in B(G)$,
        \[ (i_b)^\sharp \mathcal{F} \in \Rep(G_b(E), \Lambda)^{\geq \langle 2\rho, \nu_b \rangle}, \]
        where $(i_b)^\sharp$ denotes the right adjoint to $(i_b)_!$.
    \end{enumerate}
\end{dfn}

Let $\Loc_{^L G, E}^{\mathrm{unip}}$ denote the moduli stack of unipotent L-parameters valued in the Langlands dual group $^L G$. The categorical local Langlands correspondence intertwines the spectral action on the moduli of parameters with an action on sheaves. The $\infty$-category $\Perf(\Loc_{^L G, F}^{\mathrm{unip}})$ act on $\IndShv^{\mathrm{unip}}_{\mathrm{f.g.}}(\Isoc_G, \Lambda)$ via a spectral action denoted by $\star$.

\begin{dfn}\label{dfn:hecke_op_unip}
    Let $V \in \Rep(\hat{G})$ be a finite-dimensional representation of the dual group. Let $\widetilde{V}$ be the vector bundle on $\Loc_{^L G, E}^{\mathrm{unip}}$ associated to $V$. The \emph{Hecke operator} $T_V$ is defined as the functor:
    \[
    T_V: \Shv^{\mathrm{unip}}(\Isoc_G, \Lambda) \to \Shv^{\mathrm{unip}}(\Isoc_G, \Lambda), \quad \mathcal{F} \mapsto \widetilde{V} \star \mathcal{F}.
    \]
\end{dfn}

\begin{thm}\cite[Theorem 1.8]{yang2025generic}\label{thm:main_unip}
    Assume that the parameter $\xi$ is a generic conjugacy class of unramified $L$-parameters.
    \begin{enumerate}
        \item The unipotent local Langlands functor restricts to an equivalence of categories:
        \[
        \mathbb{L}_{G, \xi}^{\mathrm{unip}}: \IndShv^{\mathrm{unip}}_{\mathrm{f.g.}}(\Isoc_G, \Lambda)_\xi \xrightarrow{\sim} \IndCoh(\Loc_{^L G, E}^{\mathrm{unip}})_\xi.
        \]
        \item Let $V$ be a tilting $\hat{G}$-representation. Then the Hecke action $T_V$ is exotic $t$-exact when restricted to the subcategory $\Shv^{\mathrm{unip}}(\Isoc_G, \Lambda)_\xi$.
        \item Moreover, if $\Lambda = \overline{\mathbb{Q}}_\ell$, the functor $\mathbb{L}_{G, \xi}^{\mathrm{unip}}$ is compatible with the exotic $t$-structure on $\Shv^{\mathrm{unip}}(\Isoc_G, \Lambda)_\xi$ and the standard $t$-structure on $\IndCoh(\Loc_{^L G, E}^{\mathrm{unip}})_\xi$.
    \end{enumerate}
\end{thm}

\subsection{Global Geometry and the Igusa Stack}
To apply the local theory to the global setting, one utilizes the geometry of the special fiber of a meta-unitary Shimura variety. Let $(G,\mu)$ be a meta-unitary Shimura variety and $\mu^*$ be the Galios average of the cocharacter $\mu$. Let $K\subset G(A_\BQ^f)$ be a compact open subgroup such that $K_p=\CG$ is parahoric. Let $(\CS_{G,\mu,K})_{\red}$ denote the perfection of the special fiber of the integral model of the Shimura variety.

\begin{rmk}\label{rmk:2_cartesian_unip}
For a meta-unitary Shimura variety, we have the $2$-cartesian diagram:
\begin{equation}
\begin{tikzcd}
(\CS_{G,\mu,K})_{\red} \arrow[r, "\pi_{\mathrm{Crys}}"] \arrow[d, "\mathrm{Nt}"]  & \Sht_{\CG}^{\mathrm{loc}} \arrow[d, "\mathrm{Nt}"] \\
\Igs_{G,\mu,K^p} \arrow[r, "\bar{\pi}_{\mathrm{Crys}}"] & \Isoc_{G, \leq \mu^*}
\end{tikzcd}
\end{equation}
\end{rmk}

\begin{dfn}\label{dfn:canonical_duality}
    Let $\mathbb{D}_{\Igs}^{\textup{can}}$ denote the canonical duality on $\Shv(\Igs, \Lambda)$. We define the object $\omega_{\Igs}^{\textup{can}} \in \Shv(\Igs, \Lambda)$ as the admissible dual of the dualizing sheaf:
    \[
    \omega_{\Igs}^{\textup{can}} := (\mathbb{D}_{\Igs}^{\textup{can}})^{\Adm}(\omega_{\Igs}).
    \]
    Furthermore, we define the \emph{!-Igusa sheaf} $\mathcal{J}^{\textup{can}} \in \Shv(\Isoc_{G, \leq \mu^*}, \Lambda)$ by
    \[
    \mathcal{J}^{\textup{can}} := (\overline{\pi_{\mathrm{Crys}}})_\flat \omega_{\mathrm{Igs}}^{\textup{can}},
    \]
    where $(\overline{\pi_{\mathrm{Crys}}})_\flat$ is the right adjoint of $(\overline{\pi_{\mathrm{Crys}}})_!$.
\end{dfn}

One can compute the cohomology of the Igusa stack in terms of these sheaves. The proof of \cite{yang2025generic} carries over verbatim; in fact, it simplifies in our case because meta-unitary Shimura varieties are proper, and we do not need to address well-positionedness. In the computation of the stalk of the Igusa stack, they also use the existence of a $v$-sheaf local model diagram, primarily in Lemma 4.46, which holds in our situation as well.

Note that they use the fact that the local model admits a schematic model (which is true in our situation by the work of Gleason), but they do not require that the entire local model diagram admits a scheme-theoretic model.

\begin{pro}\cite[Proposition 1.9]{yang2025generic}\label{pro:cohomology_Igusa}
    For $b \in B(G, \mu)$, there is a natural isomorphism
    \[
    (i_b)^! \mathcal{J}^{\textup{can}} \simeq R\Gamma(\mathrm{Ig}_b, \Lambda)
    \]
    in $\Shv(\Isoc_{G, b}, \Lambda) \simeq \Rep(G_b(\mathbb{Q}_p), \Lambda)$, where $i_b: \Isoc_{G, b} \hookrightarrow \Isoc_{G, \leq \mu}$ denotes the inclusion of the Newton stratum.
\end{pro}
\subsection{Local-Global Compatibility}
A key innovation in the work of Yang and Zhu is the computation of the étale cohomology of Shimura varieties in terms of coherent cohomology on the stack of local Langlands parameters. This requires the definition of spectral objects corresponding to the geometric Igusa sheaves defined above.

We define the unipotent spectral Igusa sheaf by applying the unipotent local Langlands functor to the geometric !-Igusa sheaf. Let $$\Psi^L: \Shv(\Isoc_G, \Lambda) \hookrightarrow \IndShv_{\mathrm{f.g.}}(\Isoc_G, \Lambda)$$ be the natural fully faithful embedding, and let $\mathcal{P}^{\unip}$ be the right adjoint of the natural embedding $\IndShv^{\unip}_{\mathrm{f.g.}}(\Isoc_G, \Lambda) \to \IndShv_{\mathrm{f.g.}}(\Isoc_G, \Lambda)$.

\begin{dfn}\label{dfn:spectral_Igusa}
    The \emph{unipotent coherent !-Igusa sheaf} is defined as
    \[
    \mathcal{J}_{\spec}^{\can, \unip} := \mathbb{L}_G^{\unip}(\mathcal{P}^{\unip} \circ \Psi^L((i_{\leq \mu^*})_* \mathcal{J}^{\can})) \in \IndCoh(\Loc_{^L G, \mathbb{Q}_p}^{\unip}).
    \]
\end{dfn}

To state the compatibility theorem, we require two additional ingredients on the spectral side:
\begin{enumerate}
    \item Let $V_\mu$ be the highest weight representation of $\hat{G}$ associated to $\mu$. We denote by $\widetilde{V}_\mu$ the vector bundle on $\Loc_{^L G, \mathbb{Q}_p}^{\unip}$ associated to $V_\mu$.
    \item Let $\CohSpr_{^L G}^{\unip}$ be the \emph{unipotent coherent Springer sheaf} defined as
    \[
    \CohSpr_{^L G}^{\unip} := (\mathfrak{q}^{\unip})_* \omega_{\Loc_{L B, \mathbb{Q}_p}^{\unip}} \in \Coh(\Loc_{^L G, \mathbb{Q}_p}^{\unip}),
    \]
    where $\mathfrak{q}^{\unip}: \Loc_{L B, \mathbb{Q}_p}^{\unip} \to \Loc_{^L G, \mathbb{Q}_p}^{\unip}$ is the natural morphism from the stack of unipotent parameters for the dual Borel.
\end{enumerate}

The coherent Springer sheaf carries a natural action of the Iwahori-Hecke algebra $\mathcal{H}_I$, while the vector bundle $\widetilde{V}_\mu$ admits a tautological action of the Weil group $W_E$. On the geometric side, the cohomology of the Shimura variety carries actions of the prime-to-$p$ Hecke algebra $\mathcal{H}_{K^p}$, the Iwahori-Hecke algebra $\mathcal{H}_I$, and the Galois group (via $W_E$).

The following theorem establishes the local-global compatibility, enabling the computation of étale cohomology via coherent sheaves. In the proof of local–global compatibility, they also use the existence of the v-sheaf local model diagram, which likewise holds in our situation.

\begin{thm}\cite[Theorem 1.10]{yang2025generic}\label{thm:local_global}
    There is an $\mathcal{H}_{K^p} \times \mathcal{H}_I \times W_E$-equivariant isomorphism
    \[
    \Hom(\widetilde{V}_\mu \otimes \CohSpr_{^L G}^{\unip}, \mathcal{J}_{\spec}^{\can, \unip}) \simeq R\Gamma_c(\Sh_{K^p I}(G, X)_{\overline{E}}, \Lambda)(d/2)[d].
    \]
\end{thm}
\begin{thm}\label{thm:main_vanishing}
Let $(G,X)$ be a meta-unitary Shimura datum, with level $K \subset G(\BA_f)$ that is hyperspecial at $p$. Let $p > 2$ be an unramified prime. Assume that either $\Lambda = \BQ_\ell$ or $\Lambda = \BF_\ell$ with $\ell$ bigger than the Coxeter number of any simple factors of $G^\ad$. Assume that $\xi$ is generic semisimple unramified $L$-parameter. Then $R\Gamma(\Sh_K(G, X)_{\overline{E}}, \Lambda)_\xi$ is concentrated in degree $d$, where $d$ is the dimension of the Shimura variety.
\end{thm}
\bibliography{myref}{}
\bibliographystyle{alpha}
\end{document}